\documentclass[11pt]{article}
\usepackage[]{geometry}
\usepackage{graphicx}
\usepackage{caption}
\usepackage{subcaption}
\usepackage{hyperref}
\usepackage{amsmath}
\usepackage{amsthm}
\usepackage{amssymb}
\usepackage{authblk}

\usepackage{bm}
\usepackage{xcolor}
\usepackage{float}
\usepackage{booktabs}

\usepackage{yhmath}
\usepackage{mathdots}
\usepackage{MnSymbol}

\everymath{\displaystyle}
\allowdisplaybreaks

\newtheorem{theorem}{Theorem}[section]
\newtheorem{definition}{Definition}[section]

\newtheorem{remark}{Remark}[section]
\newtheorem{corollary}{Corollary}[section]
\newtheorem{lemma}{Lemma}[section]
\DeclareMathOperator*{\diag}{diag}
\DeclareMathOperator*{\supp}{supp}
\newcommand{\opt}{\mathrm{opt}}
\newcommand{\dd}{\mathrm{d}}
\newcommand{\ii}{\mathrm{i}}
\newcommand{\bi}{b}
\newcommand{\nel}{n_{\mathrm{el}}}

\providecommand{\keywords}[1]
{ \small \textbf{\textit{Keywords---}}#1 }

\title{Outlier-free isogeometric discretizations for Laplace eigenvalue problems: closed-form eigenvalue and eigenvector expressions}
\author[$\;$]{}

\affil[$\;$]{ {\large Noureddine LAMSAHEL}}
\affil[$\;$]{ {\scriptsize  Mohammed VI Polytechnic University, The UM6P Vanguard Center, Benguerir 43150, Lot 660, Hay Moulay Rachid, Morocco.}}
\affil[$\;$]{ {\scriptsize  University of Littoral C\^ote d'Opale, Laboratory of Pure and Applied Mathematics, 50 Rue F. Buisson, 62228 Calais-Cedex, France.}}
\affil[$\;$]{ {\scriptsize Emails: noureddine.lamsahel@um6p.ma, noureddine.lamsahel@etu.univ-littoral.fr  }}
\affil[$\;$]{ {\large }}
\affil[$\;$]{ {\large Carla MANNI}}
\affil[$\;$]{ {\scriptsize  University of Rome Tor Vergata, Department of Mathematics,
Via della Ricerca Scientifica 1, 00133 Rome, Italy.}}
\affil[$\;$]{ {\scriptsize Email: manni@mat.uniroma2.it}}
\affil[$\;$]{ {\large }}
\affil[$\;$]{ {\large Ahmed RATNANI}}
\affil[$\;$]{ {\scriptsize  Mohammed VI Polytechnic University, The UM6P Vanguard Center, Benguerir 43150, Lot 660, Hay Moulay Rachid, Morocco.}}
\affil[$\;$]{ {\scriptsize Email: Ahmed.RATNANI@um6p.ma}}
\affil[$\;$]{ {\large }}
\affil[$\;$]{ {\large Stefano SERRA-CAPIZZANO}}
\affil[$\;$]{ {\scriptsize University of Insubria, Department of Science and High Technology,
Via Valleggio 11, 22100 Como, Italy.}}
\affil[$\;$]{ {\scriptsize Uppsala University, Department of Information Technology, Division of Scientific Computing, ITC, L\"agerhyddsv.~2, Hus~2, P.O.~Box~337, SE-751 05 Uppsala, Sweden.}}
\affil[$\;$]{ {\scriptsize Emails: s.serracapizzano@uninsubria.it,  stefano.serra@it.uu.se}}
\affil[$\;$]{ {\large }}
\affil[$\;$]{ {\large Hendrik SPELEERS}}
\affil[$\;$]{ {\scriptsize  University of Rome Tor Vergata, Department of Mathematics,
Via della Ricerca Scientifica 1, 00133 Rome, Italy.}}
\affil[$\;$]{ {\scriptsize Email: speleers@mat.uniroma2.it}}

\date{}

\begin{document}

\maketitle

\begin{abstract}
We derive explicit closed-form expressions for the eigenvalues and eigenvectors of the matrices resulting from isogeometric Galerkin discretizations based on outlier-free spline subspaces for the Laplace operator, under different types of homogeneous boundary conditions on bounded intervals.
For optimal spline subspaces and specific reduced spline spaces, represented in terms of B-spline-like bases, we show that the corresponding mass and stiffness matrices exhibit a Toeplitz-minus-Hankel or Toeplitz-plus-Hankel structure. Such matrix structure holds for any degree $p$ and implies that the eigenvalues are an explicitly known sampling of the spectral symbol of the Toeplitz part. Moreover, by employing tensor-product arguments, we extend the closed-form property of the eigenvalues and eigenvectors to a $d$-dimensional box. 
As a side result, we have an algebraic confirmation that the considered optimal and reduced spline spaces are indeed outlier-free. 
\end{abstract}

\noindent\keywords{Laplace eigenvalue problem, Isogeometric analysis, Optimal spline subspace, Cardinal B-spline, Mass matrix, Stiffness matrix, Toeplitz matrix, Hankel matrix.}

\section{Introduction}
Isogeometric analysis (IGA) is a well-established paradigm for the numerical treatment of differential problems, introduced in \cite{Hughes:2005}, with the aim of unifying computer aided design (CAD) and finite element analysis (FEA). 
The leading idea in IGA is to use the same representation tools for the design as well as for the analysis (in an isoparametric environment), providing a true design-through-analysis methodology \cite{Cottrell:2009,Hughes:2005}.
In its original form, IGA adopts the core functions of CAD systems, i.e., B-splines and their rational counterpart (NURBS), to approximate the solution of the considered differential problems.
The isogeometric approach shows important advantages over classical $C^0$ FEA. In particular, the inherently high smoothness of B-splines and NURBS allows for a higher accuracy per degree of freedom; see \cite{Bressan:2019,Sande:2019,Sande:2020} and references therein.

Spectral analysis typically provides a global picture of the eigenvalues and eigenfunctions of a given (discretized) differential operator. The study of the spectral error is of major interest in applications because, in practice, computations are not performed in the limit of mesh refinement, and for a large class of boundary and initial-value problems the total discretization error on a given mesh can be recovered from its spectral error; see, e.g., \cite{Hughes:2014}.
Primarily second order eigenvalue problems have been addressed and related IGA discretizations have been investigated by several authors; see \cite{Cottrell:2006,garoni2019symbol,manni2022application} and references therein. 

The explicit error estimates given in \cite{Sande:2019,Sande:2020} for spline spaces of different smoothness show that discretizations in spline spaces of maximal smoothness on uniform grids are preferable. This theoretical finding totally agrees with the numerical observations: maximally smooth spline discretizations of degree $p$ on uniform grids lead to a prime discrete spectrum, which almost entirely converges to the continuous one for increasing $p$ and does not present the well-known artifacts of high-order $C^0$ FEA, the so-called {optical branches}.
From a spectral and matrix theoretical point of view \cite{Ba2,SIMAX-Q-p-d,garoni2019symbol}, it is clearly understood that the number of optical branches amounts to $(p-k)^d-1$, with $k$ being the imposed smoothness and $d$ the dimensionality of the differential problem. The previous formula indicates that these spectral artifacts disappear only in the case of maximal smoothness, i.e., $k=p-1$, while in all other cases their number is exponential in the dimensionality $d$.

Notwithstanding maximally smooth spline spaces on uniform grids are an excellent choice for addressing eigenvalue problems, they are not completely satisfactory:
a very small portion of the eigenvalues are poorly approximated and the corresponding computed values are much larger than the exact ones. These spurious values are usually referred to as {outliers} \cite{Cottrell:2006}. 
The number of outliers increases with the degree $p$ 
and, for fixed $p$, it is independent of the number of degrees of freedom for univariate problems, while a ``thin layer'' of outliers appears in the multivariate setting. Discretization techniques that remove those outliers are often called outlier-free discretizations.
Accurate outlier-free discretizations are important both for their superior description of the spectrum of the continuous operator and for their beneficial effects in various applicative contexts such as (explicit) dynamics.

A solid analytical answer to the problem of outlier removal from isogeometric Galerkin discretizations of second order eigenvalue problems has been recently given in \cite{manni2022application} by considering subspaces of maximally smooth splines of any degree that are optimal spaces in the sense of $L^2$ Kolmogorov $n$-widths for certain function classes related to the considered eigenvalue problems \cite{Floater:2017,Floater:2018}. These optimal spline subspaces are simply identified by additional homogeneous boundary conditions and by knots that are uniform (or shifted uniform) and are equipped with a B-spline-like basis. For any degree, they give rise to a complete removal of the outliers without loss of accuracy in the approximation of all eigenfunctions.

The fact that outliers are related to the treatment of boundary conditions has been observed already for a long time. Interesting empirical procedures for outlier removal based on this intuition have been presented in \cite{deng2021boundary,hiemstra2021removal}, where the authors exploit the imposition of suitable homogeneous additional boundary conditions in a similar (but slightly different) manner to the above mentioned optimal spline subspaces. More precisely,
in \cite{deng2021boundary} a penalization of specific high-order derivatives near the
boundary is proposed, while in \cite{hiemstra2021removal} the same high-order derivatives at the boundary are strongly set equal to zero, by using suitable reduced spline spaces as trial spaces; see also \cite{sogn2019robust,takacs2016approximation}. Although there is a clear numerical evidence that the strong imposition of the considered additional boundary conditions removes the outliers without affecting the accuracy of the approximation for the global spectrum, a proper theoretical foundation of the whole process is missing.
 
The	knowledge of the spectrum of matrices arising from a given discretization technique for partial differential equations, not necessarily confined to eigenvalue differential problems, plays also an important role in designing efficient and robust solvers for the resulting linear systems. It is the foundation for the convergence analysis of (preconditioned) Krylov methods and for the theoretical analysis of effective preconditioners and fast multigrid/multi-iterative solvers. 

The spectral behavior of matrices arising from isogeometric Galerkin methods based on standard B-splines and NURBS, with boundary conditions in strong or weak form, has been analyzed in a series of recent papers; see, e.g., \cite{garoni2022,garoni2020NURBS} and references therein. 
The analysis typically exploits the theory of generalized locally Toeplitz (GLT) sequences \cite{garoni2017generalized,garoni2018generalized} as well as properties of the bases used to represent the discretization spaces. More precisely, it has been shown that sequences of such matrices corresponding to different levels of mesh refinement admit an eigenvalue distribution, which can be compactly described by a function called (GLT) symbol. 
This symbol has a canonical structure incorporating the approximation technique, the geometry of the domain, and the coefficients of the principal terms of the differential problem. The eigenvalues of the considered matrices can be approximated by a uniform sampling of the symbol and, when the size of the matrices increases, such an approximation converges to the true spectrum, up to few outliers. 
As application, the spectral information obtained from the symbol has been utilized to design efficient preconditioners and multigrid solvers \cite{SINUM}, but the presence of the outliers dirties and weakens the elegant symbol information. 

We want to emphasize that none of the above results provides closed-form expressions of the discrete eigenvalues and the corresponding eigenvectors for the considered (full or reduced) spline discretization spaces, unless in few specific cases of low degree \cite{ekstrom2018eigenvalues}.
However, it has been observed in \cite{deng2021analytical} that the discretization matrices for the standard full spline spaces are almost Toeplitz-minus-Hankel or Toeplitz-plus-Hankel matrices and that, if they would have exactly this structure, explicit closed-form expressions for the corresponding eigenpairs could be realized. Such structured matrices belong to certain $\tau$ matrix algebras, which have been widely studied in the literature; see, e.g., \cite{bini1983spectral,bozzo1995use,di1995matrix} and references therein.

In this paper, we continue the spectral study of isogeometric Galerkin discretizations in optimal spline subspaces for the Laplace operator, under different types of boundary conditions, and we assume that the considered subspaces are represented in terms of the B-spline-like basis from \cite{manni2022application} (see also \cite{DiVona:2019}).
We derive explicit closed-form expressions for the eigenvalues and eigenvectors of the corresponding mass and stiffness matrices, for any degree. These expressions show that the eigenvalues are a simple sampling of the GLT symbol \cite{garoni2014spectrum}, with no outliers. Moreover, when fixing the type of boundary conditions, both the mass and the stiffness matrix exhibit the same Toeplitz-minus-Hankel or Toeplitz-plus-Hankel structure, and thus they belong to the same $\tau$ matrix algebra \cite{bozzo1995use}. This allows us to deduce also explicit closed-form expressions for the approximated spectrum of the Laplace operator, under the addressed boundary conditions. In this way, we get an algebraic confirmation that optimal spline subspaces lead to outlier-free discretizations.

Furthermore, we extend our analysis to isogeometric Galerkin discretizations based on the {reduced spline spaces} considered in \cite{hiemstra2021removal} (see also \cite{sogn2019robust,takacs2016approximation}) and represented in terms of the B-spline-like basis from \cite{manni2022application}. It turns out that the corresponding stiffness and mass matrices exhibit again a Toeplitz-minus/plus-Hankel structure. This gives rise to explicit closed-form expressions for their eigenvalues and eigenvectors. As a consequence, we can conclude that such spaces are outlier-free as well.

\subsection{Main results}
The eigenvalues of a square matrix $A$ of size $n$ are denoted by $\lambda_j(A)$, $j=1,\ldots,n$ and the corresponding eigenvectors by $\bm{u}_j(A)$ with components $u_{i,j}(A)$, $i,j=1,\ldots,n$.
We define
\begin{equation*}
g_p^r(\theta)= (-1)^{r}\mathcal{N}_{2p+1}^{(2r)}(p+1)+2(-1)^{r}\displaystyle\sum_{k=1}^p\mathcal{N}^{(2r)}_{2p+1}(p+1-k)\cos(k\theta),\quad \theta\in[0,\pi], \quad r\in\{0,1\}, \quad p\geq1,
\end{equation*}
where $\mathcal{N}_p$ is the cardinal B-spline of degree $p\in\mathbb{N}$, given by \eqref{eq:cardinal-B-spline-1} and \eqref{eq:cardinal-B-spline-2}.
Then, our main results can be summarized as follows.

\begin{enumerate}
\item We derive explicit closed-form expressions for the eigenvalues and eigenvectors of the mass matrix $M_n^{p,\bi}$ and the stiffness matrix $K_n^{p,\bi}$ when considering the B-spline-like basis of the univariate optimal spline subspaces from \cite{manni2022application} of dimension $n$, for any degree $p\geq1$ and three different types of boundary conditions, indicated with the index $\bi=0,1,2$.
\begin{itemize}
\item Dirichlet boundary conditions \eqref{Dbc}, indicated with $\bi=0$: The matrices $M_n^{p,0}$ and $K_n^{p,0}$ are symmetric and centrosymmetric. Under the assumption $n\geq \max\left\{p+1,p+\left\lfloor\frac{p}{2}\right\rfloor-1\right\}$, Theorem~\ref{closedformdiri} and Corollary~\ref{corollaryeigdiri} reveal that
\begin{equation*}
\lambda_j\left(M_n^{p,0}\right)=(n+1)^{-1}\;g^{0}_p\left(\frac{j\pi}{n+1}\right),\quad\lambda_j\left(K_n^{p,0}\right)=(n+1)\;g^{1}_p\left(\frac{j\pi}{n+1}\right),\quad j=1,\ldots,n,
\end{equation*}
and
\begin{equation*}
u_{i,j}\left(M_n^{p,0}\right)=u_{i,j}\left(K_n^{p,0}\right)=\sqrt{\frac{2}{n+1}}\sin\left(\frac{ij\pi}{n+1}\right),\quad i,j=1,\ldots, n.
\end{equation*}
\item Neumann boundary conditions \eqref{Nbc}, indicated with $\bi=1$:
The matrices $M_n^{p,1}$ and $K_n^{p,1}$ are symmetric and centrosymmetric. Under the assumption $n\geq \max\left\{2p-\left\lfloor\frac{p}{2}\right\rfloor,2p-2\left\lfloor\frac{p}{2}\right\rfloor+1\right\}$, Theorem~\ref{closedformneum} and Corollary~\ref{corollaryeigneum} reveal that
\begin{equation*}
\lambda_j\left(M_n^{p,1}\right)=n^{-1}\;g^{0}_p\left(\frac{(j-1)\pi}{n}\right),\quad\lambda_j\left(K_n^{p,1}\right)=n \;g^{1}_p\left(\frac{(j-1)\pi}{n}\right),\quad j=1,\ldots,n,
\end{equation*}
and
\begin{equation*}
u_{i,j}\left(M_n^{p,1}\right)=u_{i,j}\left(K_n^{p,1}\right)=\sqrt{\frac{2}{n}}\;c_j\cos\left(\frac{(j-1)\pi}{n}\left(i-\frac{1}{2}\right)\right),\quad c_j=\left\{\begin{array}{ll}
    \frac{1}{\sqrt{2}},& j=1, \\
    1,& \text{otherwise},
\end{array}\quad i,j=1,\ldots, n.
\right.
\end{equation*}
\item Mixed boundary conditions \eqref{Mbc}, indicated with $\bi=2$:
The matrices $M_n^{p,2}$ and $K_n^{p,2}$ are only symmetric. Under the assumption $n\geq \max\left\{p+1,p+\left\lfloor\frac{p}{2}\right\rfloor\right\}$, Theorem~\ref{closedformmixed} and Corollary~\ref{corollaryeigmixed} reveal that
\begin{equation*}
\lambda_j\left(M_n^{p,2}\right)=\left(\frac{2n+1}{2}\right)^{-1}g^{0}_p\left(\frac{(2j-1)\pi}{2n+1}\right),\quad \lambda_j\left(K_n^{p,2}\right)=\left(\frac{2n+1}{2}\right)g^{1}_p\left(\frac{(2j-1)\pi}{2n+1}\right),\quad j=1,\ldots,n,
\end{equation*}
and
\begin{equation*}
u_{i,j}\left(M_n^{p,2}\right)=u_{i,j}\left(K_n^{p,2}\right)=\sqrt{\frac{4}{2n+1}}\sin\left(\frac{i(2j-1)\pi}{2n+1}\right),\quad  i,j=1,\ldots, n.
\end{equation*}
We also identify a new Hankel matrix for which we derive the explicit eigenvalues and eigenvectors in Proposition~\ref{closedformpropmixd}.
\end{itemize}
    
\item The above results for optimal spline subspaces are expanded to other outlier-free spline spaces. For simplicity of presentation, we concentrate exclusively on Dirichlet boundary conditions but similar results hold for other boundary conditions as well.

\begin{itemize}
\item We derive explicit closed-form expressions for the eigenvalues and eigenvectors of the mass matrix $\overline{M}_n^{p,0}$ and the stiffness matrix $\overline{K}_n^{p,0}$, when considering the B-spline-like basis of the univariate reduced spline spaces from \cite{manni2022application} of dimension $n$, for any even degree $p\geq1$ and Dirichlet boundary conditions.
The matrices $\overline{M}_n^{p,0}$ and $\overline{K}_n^{p,0}$ are symmetric and centrosymmetric.
Under the assumption $n\geq p+\frac{p}{2}$, Theorem~\ref{closedformreduced} and Corollary~\ref{corollaryeigreduced} reveal that
\begin{equation*}
\lambda_j\left(\overline{M}_n^{p,0}\right)=n^{-1}\;g^{0}_p\left(\frac{j\pi}{n}\right),\quad \lambda_j\left(\overline{K}_n^{p,0}\right)=n\;g^{1}_p\left(\frac{j\pi}{n}\right),\quad j=1,\ldots,n,
\end{equation*}
and
\begin{equation*}
u_{i,j}\left(\overline{M}_n^{p,0}\right)=u_{i,j}\left(\overline{K}_n^{p,0}\right)=\sqrt{\frac{2}{n}}\;c_j\sin\left(\frac{j\pi}{n}\left(i-\frac{1}{2}\right)\right),\quad    c_j=\left\{\begin{array}{ll}
    \frac{1}{\sqrt{2}},& j=n, \\
    1,& \text{otherwise},
\end{array}\right.  \quad  i,j=1,\ldots, n.
\end{equation*}
For other degrees $p$ and other boundary conditions, the reduced spline spaces from \cite{hiemstra2021removal} are either optimal spline subspaces, and hence covered by item 1, or the analysis is similar.
    
\item We extend the closed-form property of the eigenvalues and eigenvectors to the multivariate tensor-product setting on a $d$-dimensional box; see Theorem~\ref{extendclosedform}.
\end{itemize}
\end{enumerate}
In all cases, explicit closed-form expressions for the resulting approximations of the spectrum of the Laplace operator, under the addressed boundary conditions, can be immediately deduced; see Corollaries~\ref{coroloryclosedform0},~\ref{coroloryclosedform1},~\ref{coroloryclosedform2}, and~\ref{coroloryclosedformreduced}.
Our main results confirm that the optimal spline subspaces from \cite{manni2022application} and the reduced spline spaces from \cite{hiemstra2021removal} are indeed outlier-free; see Remarks~\ref{rmk:outlierfree0},~\ref{rmk:outlierfree1},~\ref{rmk:outlierfree2}, and~\ref{rmk:outlierfreereduced}. Moreover, it follows that, for degree $p=3$ (resp., $p=2$) and under Dirichlet (resp., Neumann) boundary conditions, the penalization technique in \cite{deng2021boundary}, when the penalty parameter goes to infinity, exactly corresponds to the first (resp., second) optimal spline subspace discretization.
Finally, it is also evident that the GLT symbols of the optimal spline subspaces and the above reduced spline spaces are identical and equal to the symbol of the related full spline space of the same degree.

\subsection{Outline of the paper}
In Section~\ref{sec:preliminaries}, we provide the necessary notations, definitions, and preliminary results crucial for our analysis. We start by introducing the eigenvalue problems of interest and by providing a brief overview of optimal spline subspaces designed for three classical types of boundary conditions (Dirichlet/Neumann/mixed) along with the construction of the corresponding B-spline-like bases. Afterwards,  we collect from the literature some results about the closed form of the eigenvalues and eigenvectors for structured matrices of Toeplitz-minus-Hankel or Toeplitz-plus-Hankel form.  In Section~\ref{sec:closedformoptspaces}, we delve into the matrices resulting from isogeometric Galerkin discretizations based on optimal spline subspaces. This section is organized into three subsections, each dedicated to a specific type of boundary conditions. Adopting a coherent format across these subsections, we initiate the discussion by introducing a matrix notation designed to handle both mass and stiffness matrices simultaneously. Upon establishing the Toeplitz-minus-Hankel or Toeplitz-plus-Hankel structure of this matrix, we deduce a closed-form expression for its eigenvalues and eigenvectors. For mixed boundary conditions, a new class of Hankel matrices is identified and the expression for the associated eigenvalues and eigenvectors is derived. To enhance readability, the detailed proofs of the main results presented in this section are collected in Appendix~\ref{sec:A}. In Section~\ref{sec:otherspaces}, our focus shifts to other outlier-free spline spaces. Firstly, we consider specific reduced spline spaces, similar to the optimal spline subspaces. Following a parallel approach to that in Section~\ref{sec:closedformoptspaces}, we derive closed-form expressions for the eigenvalues and eigenvectors of the corresponding mass and stiffness matrices, which allows us to conclude that these spaces are truly outlier-free. Then, we extend the closed-form expressions to spline discretizations defined on a bounded box of $\mathbb{R}^{d}$, $d\geq1$. Finally, Section~\ref{sec:conclusions} summarizes the results, provides some further comments, and discusses some future work.

\section{Model problems and preliminaries}\label{sec:preliminaries}
In this section, we describe our eigenvalue problems of interest, considering three classical types of boundary conditions and for each of them we define a family of optimal spline subspaces. In \cite{manni2022application}, it was shown that these optimal spline subspaces lead to outlier-free discretizations and a B-spline-like basis was constructed, which we summarize here. As we will show later in the paper, the corresponding discretization matrices possess a Toeplitz-minus-Hankel or Toeplitz-plus-Hankel structure. Therefore, we present some known results about the closed form of the eigenvalues and eigenvectors of such matrices here as well.

\subsection{Eigenvalue problems}
We consider three one-dimensional eigenvalue problems, each corresponding to the Laplace operator, of the form
\begin{equation}\label{pd1}
\left\{
\begin{array}{ll}
-u''=\lambda u,\quad \text{in } (0,1),\\[0.1cm]
\mathbf{BC}.
\end{array}
\right.
\end{equation}
Here, $\mathbf{BC}$ serves as an abbreviation for boundary conditions and represents one of the following.
\begin{itemize}
\item Dirichlet boundary conditions:
\begin{equation}\label{Dbc}
u(0)=u(1)=0;
\end{equation}
\item Neumann boundary conditions:
\begin{equation}\label{Nbc}
u'(0)=u'(1)=0;
\end{equation}
\item Mixed boundary conditions:
\begin{equation}\label{Mbc}
u(0)=u'(1)=0.
\end{equation}
\end{itemize}

It is well established that the families of non-trivial exact solutions $(\lambda_k,u_k)$ of \eqref{pd1} subject to the boundary conditions \eqref{Dbc}--\eqref{Mbc} are given by, respectively,
\begin{align}
&\left\{ \left( \left(k\pi\right)^2,\; \sin(k\pi \;\cdot) \right),\quad  k=1,2,\ldots\, \right\}, \label{exactDbc} \\
&\left\{ \left( \left((k-1)\pi\right)^2,\; \cos((k-1)\pi \;\cdot) \right),\quad  k=1,2,\ldots\, \right\}, \label{exactNbc} \\
& \left\{ \left(\left( \frac{(2k-1)\pi }{2}\right)^2,\; \sin\left(\frac{(2k-1)\pi}{2} \;\cdot\right) \right),\quad  k=1,2,\ldots\, \right\}. \label{exactMbc}
\end{align}

In \cite[Section~3]{ekstrom2018eigenvalues}, it was illustrated how a closed-form expression for the approximate eigenvalues and eigenvectors can be achieved for Galerkin discretizations in standard full spline spaces related to the Dirichlet boundary conditions \eqref{Dbc}. However, the proposed approach is restricted to $p=1,2$. When $p$ is greater than or equal to $3$, the mass and stiffness matrices are not in the same matrix algebra anymore, making the approach unfeasible. More recently, in \cite[Section~5]{deng2021boundary}, it was analytically shown that a closed-form expression for the approximate eigenvalues and eigenvectors can be derived using another isogeometric discretization method, namely the penalization technique. Similar to the limitation in \cite{ekstrom2018eigenvalues}, the exact expression is confined to specific scenarios, namely $p=3$ for Dirichlet boundary conditions \eqref{Dbc} and $p=2$ for Neumann boundary conditions \eqref{Nbc}.

One of the main objectives of this paper is to derive an explicit closed-form expression for the approximate eigenvalues and eigenvectors of the pairs in \eqref{exactDbc}--\eqref{exactMbc} when an optimal spline subspace (or other types of outlier-free spline spaces) is used as a Galerkin discretization space, for any degree $p\geq1$.

\begin{remark} \label{periodic}
Other classical boundary conditions for the eigenvalue problem \eqref{pd1} are {periodic boundary conditions}. In this case, the discretization based on classical periodic B-splines on uniform grids gives rise to a circulant structure both for the stiffness and the mass matrix, and thus results in a full (explicit) knowledge of their spectrum; see, e.g.,~\cite{takacs2016approximation}.
\end{remark}

\subsection{Optimal spline subspaces}
We provide a concise overview of the optimal spline discretizations associated with \eqref{pd1} and presented in \cite{manni2022application}. We start by defining the optimal spline subspaces, followed by the corresponding variational formulation of the discrete systems. Lastly, we describe the construction of a B-spline-like basis for each subspace in terms of cardinal B-splines. For more details on these spaces, we refer the reader to \cite{manni2022application} and references therein.

Let $p,\nel\in\mathbb{N}^{\star}$. Following \cite{manni2022application} and using the same notation, we consider the vector of strictly increasing breakpoints
\begin{equation}\label{orgbreakp}
\bm{\tau}=( \tau_0=0,\tau_1,\ldots,\tau_{\nel-1},\tau_{\nel}=1),
\end{equation}
where $\nel$ is the number of intervals of $[0,1]$. We then introduce the standard spline space of maximal smoothness 
\begin{equation*}
\mathbb{S}_{p,\bm{\tau}} = \left\{ s\in\mathcal{C}^{p-1}([0,1]):\; s|_{\left[ \tau_i, \, \tau_{i+1} \right)}\in \mathbb{P}_{p},\ i=0,\ldots,\nel-1 \right\},
\end{equation*}
where $\mathbb{P}_{p}$ is the space of polynomials of degree less than or equal to $p$. A comprehensive exposition of the space $\mathbb{S}_{p,\bm{\tau}}$ and its B-spline basis  can be found in \cite{de1978practical,Lyche:2018,schumaker2007spline}.

Now, let us fix an integer number $\ell\in\{0,\ldots,p\}$. In the following, we will focus on specific subspaces of $\mathbb{S}_{p,\bm{\tau}}$ designed to selectively capture elements with certain vanishing derivatives at the boundary, precisely up to order~$\ell$:
\begin{align}
\mathbb{S}_{p,\bm{\tau},0}^{\ell} &= \left\{ s\in \mathbb{S}_{p,\bm{\tau}}:\;  s^{(\beta)}(0)=s^{(\beta)}(1)=0,\  0\leq \beta\leq \ell,\ \beta\;\text{even}    \right\},
\label{S0l} \\
\mathbb{S}_{p,\bm{\tau},1}^{\ell} &= \left\{ s\in \mathbb{S}_{p,\bm{\tau}}:\; s^{(\beta)}(0)=s^{(\beta)}(1)=0,\ 0\leq \beta\leq \ell,\ \beta\;\text{odd}    \right\},
\label{S1l} \\
\mathbb{S}_{p,\bm{\tau},2}^{\ell} &= \left\{ s\in \mathbb{S}_{p,\bm{\tau}}:\; s^{(\beta_0)}(0)=s^{(\beta_1)}(1)=0,\ 0\leq \beta_0,\beta_1\leq \ell,\ \beta_0\;\text{even},\  \beta_1\;\text{odd}  \right\}.
\label{S2l}
\end{align}
It can be easily verified that the continuous eigenvectors \eqref{exactDbc}--\eqref{exactMbc} of the problems \eqref{pd1} satisfy the boundary conditions in the spaces $\mathbb{S}_{p,\bm{\tau},\bi}^{\ell}$ for $\bi=0,1,2$, respectively. The optimal spline subspaces of dimension $n$ $\left(n\in\mathbb{N}^{\star}\right)$ associated with the boundary conditions \eqref{Dbc}--\eqref{Mbc} are obtained by taking $\ell=p$ and specific breakpoints. More precisely,
\begin{equation}\label{Sopti}
\mathbb{S}_{p,n,\bi}^{\opt}=\mathbb{S}_{p,{\bm{\tau}}_{p,n,\bi}^{\opt},\bi}^{p},\quad \bi=0,1,2,
\end{equation}
where the breakpoints $\bm{\tau}_{p,n,\bi}^{\opt}$, $\bi=0,1,2$ are specified as
\begin{align*}
\bm{\tau}_{p,n,0}^{\opt} &= \left\{
\begin{array}{ll}
\left(0,\frac{1}{n+1},\frac{2}{n+1},\ldots,\frac{n}{n+1},1 \right), & p \text{ odd}, \\[0.3cm]
\left(0,\frac{1/2}{n+1},\frac{3/2}{n+1},\ldots,\frac{n+1/2}{n+1},1 \right), \ \ & p \text{ even},\\
\end{array}
\right.
\\
\bm{\tau}_{p,n,1}^{\opt} &= \left\{
\begin{array}{ll}
\left(0,\frac{1/2}{n},\frac{3/2}{n},\ldots,\frac{n-1/2}{n},1 \right), \quad\quad & p \text{ odd}, \\[0.3cm]
\left(0,\frac{1}{n},\frac{2}{n},\ldots,\frac{n-1}{n},1 \right), & p \text{ even},\\
\end{array}
\right.
\\
\bm{\tau}_{p,n,2}^{\opt} &= \left\{
\begin{array}{ll}
\left(0,\frac{2}{2n+1},\frac{4}{2n+1},\ldots,\frac{2n}{2n+1},1 \right),\, & p \text{ odd}, \\[0.3cm]
\left(0,\frac{1}{2n+1},\frac{3}{2n+1},\ldots,\frac{2n-1}{2n+1},1 \right), & p \text{ even}.\\
\end{array}
\right. 
\end{align*}
In the Galerkin method, the approximate solutions of \eqref{pd1} with \eqref{Dbc}--\eqref{Mbc} are computed as follows:
\begin{equation}\label{app}
\text{Find}\quad (\lambda_{k,h},u_{k,h})\in\mathbb{R}\times\mathbb{S}_{p,n,\bi}^{\opt} \quad \text{such that}\quad   
K(u_{k,h},v)=\lambda_{k,h}M(u_{k,h},v),\quad \forall v\in\mathbb{S}_{p,n,\bi}^{\opt},
\end{equation}
where 
\begin{equation*}
M(u_{k,h},v)=\displaystyle\int_{0}^{1} u_{k,h}(x) v(x)\;{\dd}x,\quad
K(u_{k,h},v)=\displaystyle\int_{0}^{1} u_{k,h}'(x) v'(x)\;{\dd}x.
\end{equation*}
For well-posedness of the problem, the convergence of \eqref{app}, and the outlier-free property of the spaces \eqref{Sopti}, we refer the reader to \cite[Section~4]{manni2022application}.

Given a fixed $\bi\in\{0,1,2\}$, we have $\dim\left(\mathbb{S}_{p,n,\bi}^{\opt}\right)=n$. Let $B_{\bi}^p=\left\{N^p_{1,\bi},N^p_{2,\bi},\ldots,N^p_{n,\bi}\right\}$ denote a basis of the space $\mathbb{S}_{p,n,\bi}^{\opt}$ (our basis of interest will be specified later in Sections~\ref{sec:DBCBasis}--\ref{sec:MBCBasis}). The problem described by \eqref{app} can be reformulated as a discrete eigenvalue problem,
\begin{equation}\label{systfindimi}
\left(M_n^{p,\bi}\right)^{-1}K_n^{p,\bi}\; \bm{u}_{k,h}=\lambda_{k,h}\bm{u}_{k,h},
\end{equation}
where $\bm{u}_{k,h}$ represents the coefficient vector of $u_{k,h}$ in terms of the basis $B_{\bi}^p$. The matrices $M_n^{p,\bi}$ and  $K_n^{p,\bi}$ are called the mass matrix and the stiffness matrix, respectively, and are defined by
\begin{equation}\label{allmatrices}
M_n^{p,\bi}=\left[\displaystyle\int_0^1 N^p_{i,\bi}(x) N^p_{j,\bi}(x)\;{\dd}x\right]_{i,j=1}^n,\quad K_n^{p,\bi}=\left[\displaystyle\int_0^1 (N^p_{i,\bi})'(x) (N^p_{j,\bi})'(x)\;{\dd}x\right]_{i,j=1}^n.
\end{equation}

We will conclude our summary of the optimal spline subspaces presented in \cite{manni2022application} by introducing a B-spline-like basis for each space. To this end, we need to define the cardinal B-splines and their properties; see, e.g., \cite{de1978practical,Lyche:2018} for more details.
Let $\mathcal{N}_p: \mathbb{R}\longrightarrow\mathbb{R} $ be the cardinal B-spline of degree $p$ recursively defined as
\begin{equation}\label{eq:cardinal-B-spline-1}
\mathcal{N}_0(t)=\left\{
\begin{array}{ll}
1, & t\in[0,1),\\[0.05cm]
0, & \text{otherwise},
\end{array}
\right.
\end{equation}
and
\begin{equation}\label{eq:cardinal-B-spline-2}
\mathcal{N}_p(t)=\frac{t}{p}\mathcal{N}_{p-1}(t)+\frac{p+1-t}{p}\mathcal{N}_{p-1}(t-1),\quad t\in\mathbb{R},\quad p\geq 1.
\end{equation}
The function class $\left(\mathcal{N}_p\right)_{p\geq0}$ possesses the following properties.
\begin{itemize}
\item Smoothness and local support:
\begin{equation}\label{suppandregucb}
\mathcal{N}_p\in \mathcal{C}^{p-1}(\mathbb{R}),\quad \supp\left(\mathcal{N}_p\right)= [0,p+1].
\end{equation}
\item Symmetry:
\begin{equation}\label{symmetrycb}
\mathcal{N}_p^{(r)}\left(\frac{p+1}{2}+t\right)=(-1)^r  \mathcal{N}_p^{(r)}\left(\frac{p+1}{2}-t\right),\quad r\leq p.
\end{equation}
\item The $L^2$ inner products:
\begin{equation}\label{innerprodcb}
\displaystyle\int_{\mathbb{R}}\mathcal{N}_{p_1}^{(r_1)}(t)\mathcal{N}_{p_2}^{(r_2)}(t+\rho)\;{\dd}t =(-1)^{r_1}\mathcal{N}_{p_1+p_2+1}^{(r_1+r_2)}(p_1+1+\rho)
=(-1)^{r_2}\mathcal{N}_{p_1+p_2+1}^{(r_1+r_2)}(p_2+1-\rho),
\end{equation}
for every $\rho\in\mathbb{R}$ and every $r_1\leq p_1$ and $r_2\leq p_2$.
\end{itemize}
We are now ready to discuss the construction of the B-spline-like basis $B_{\bi}^p$ corresponding to $\mathbb{S}_{p,n,\bi}^{\opt}$ for $\bi=0,1,2$ in the following subsections. For each basis we present its extraction formula from cardinal B-splines. For a more in-depth understanding, we refer the reader to \cite{DiVona:2019,manni2022application}.

\subsubsection{Dirichlet boundary conditions}\label{sec:DBCBasis}
Let us consider the first optimal spline subspace $\mathbb{S}_{p,n,0}^{\opt}$ corresponding to the boundary condition \eqref{Dbc} of our eigenvalue problem \eqref{pd1}. The B-spline-like basis of this space is denoted by $B_0^p=\left\{N^{p}_{i,0},\ i=1,\ldots,n\right\}$ and its definition depends on the parity of $p$ as follows.

In the case $p$ odd, the basis is computed as
\begin{equation}\label{basis0podd}
\begin{bmatrix}
N^p_{1,0}\\
N^p_{2,0}\\
\vdots\\
N^p_{n,0}\\
\end{bmatrix}= \begin{bmatrix}
\underbrace{\overbrace{\cdots\bigg\vert L^0_{n} \bigg\vert L^0_{n}\bigg\vert}^{\frac{p+1}{2}}  I_{n}  \overbrace{ \bigg\vert  R^0_{n}   \bigg\vert R^0_{n}    \bigg\vert   \cdots  }^{\frac{p+1}{2}}  }_{n+p+1}
\end{bmatrix}\begin{bmatrix}
N^p_{-p}\\
\vdots\\
N^p_{0}\\
N^p_{1}\\
\vdots\\
N^p_{n}\\
\end{bmatrix},
\end{equation}
where the set $\left\{ N_{i}^p,\ i=-p,\ldots,n  \right\}$
is derived from the cardinal B-spline $\mathcal{N}_p$ in \eqref{eq:cardinal-B-spline-2} as
\begin{equation}\label{bsplines0odd}
N_{i}^p(x)=\mathcal{N}_p\left((n+1)x-i\right),\quad i=-p,\ldots,n,
\end{equation}
and the matrices $L^0_n$ and $R^0_n$ are composed of the identity matrix $I_n$, the exchange matrix $J_n$, and the zero column vector $O_n$ as
\begin{equation*}
L^0_{n}=\begin{bmatrix}
I_{n} \bigg\vert O_{n} \bigg\vert -J_{n} \bigg\vert O_{n}
\end{bmatrix},\quad R^0_{n}=\begin{bmatrix}
O_{n} \bigg\vert -J_{n} \bigg\vert O_{n} \bigg\vert I_{n}
\end{bmatrix}.
\end{equation*}
In the case $p$ even, the basis is computed as
\begin{equation}\label{basis0peven}
\begin{bmatrix}
N^p_{1,0}\\
N^p_{2,0}\\
\vdots\\
N^p_{n,0}\\
\end{bmatrix}= \begin{bmatrix}
\underbrace{\overbrace{\cdots\bigg\vert L^0_{n} \bigg\vert L^0_{n}\bigg\vert}^{\frac{p}{2}+1}  I_{n}  \overbrace{ \bigg\vert  R^0_{n}  \bigg\vert R^0_{n}    \bigg\vert   \cdots  }^{\frac{p}{2}+1}  }_{n+p+2}
\end{bmatrix}\begin{bmatrix}
N^p_{-p}\\
\vdots\\
N^p_{0}\\
N^p_{1}\\
\vdots\\
N^p_{n+1}\\
\end{bmatrix},
\end{equation}
where the set $\left\{ N_{i}^p,\ i=-p,\ldots,n+1  \right\}$
is defined by
\begin{equation}\label{bsplines0even}
N_{i}^p(x)=\mathcal{N}_p\left((n+1)x-\left(i-\frac{1}{2}\right)\right),\quad i=-p,\ldots,n+1.
\end{equation}

\subsubsection{Neumann boundary conditions}\label{sec:NBCBasis}
Now, we turn our attention to the second optimal spline subspace $\mathbb{S}_{p,n,1}^{\opt}$, which is associated with the boundary condition \eqref{Nbc}. The B-spline-like basis of this space is denoted by $B_1^p=\left\{N^{p}_{i,1},\ i=1,\ldots,n\right\}$ and determined as follows, with a dependence on the parity of $p$.

In the case $p$ odd, the basis is computed as
\begin{equation}\label{basis1podd}
\begin{bmatrix}
N^p_{1,1}\\
N^p_{2,1}\\
\vdots\\
N^p_{n,1}\\
\end{bmatrix}= \begin{bmatrix}
\underbrace{\overbrace{\cdots\bigg\vert J_{n} \bigg\vert I_{n} \bigg\vert J_{n}\bigg\vert}^{\frac{p+1}{2}}  I_{n}  \overbrace{ \bigg\vert  J_{n}   \bigg\vert I_n \bigg\vert J_{n}    \bigg\vert   \cdots  }^{\frac{p+1}{2}}  }_{n+p+1}
\end{bmatrix}\begin{bmatrix}
N^p_{-p}\\
\vdots\\
N^p_{0}\\
N^p_{1}\\
\vdots\\
N^p_{n}\\
\end{bmatrix},
\end{equation}
where the set $\left\{ N_{i}^p,\ i=-p,\ldots,n  \right\}$
is given by
\begin{equation}\label{bsplines1odd}
N_{i}^p(x)=\mathcal{N}_p\left(nx-\left(i-\frac{1}{2}\right)\right),\quad i=-p,\ldots,n.
\end{equation}
In the case $p$ even, the basis is computed as
\begin{equation}\label{basis1peven}
\begin{bmatrix}
N^p_{1,1}\\
N^p_{2,1}\\
\vdots\\
N^p_{n,1}\\
\end{bmatrix}= \begin{bmatrix}
\underbrace{\overbrace{\cdots\bigg\vert J_{n} \bigg\vert I_{n} \bigg\vert J_{n}\bigg\vert}^{\frac{p}{2}}  I_{n}  \overbrace{ \bigg\vert  J_{n}   \bigg\vert I_{n}\bigg\vert J_{n}    \bigg\vert    \cdots  }^{\frac{p}{2}}  }_{n+p}
\end{bmatrix}\begin{bmatrix}
N^p_{-p}\\
\vdots\\
N^p_{0}\\
N^p_{1}\\
\vdots\\
N^p_{n-1}\\
\end{bmatrix},
\end{equation}
where the set $\left\{ N_{i}^p,\ i=-p,\ldots,n-1  \right\}$
is defined by
\begin{equation}\label{bsplines1even}
N_{i}^p(x)=\mathcal{N}_p\left(nx-i\right),\quad i=-p,\ldots,n-1.
\end{equation}

\subsubsection{Mixed boundary conditions}\label{sec:MBCBasis}
Finally, let us consider the last optimal spline subspace $\mathbb{S}_{p,n,2}^{\opt}$. The B-spline-like basis of this space is denoted by $B_2^p=\left\{N^{p}_{i,2},\ i=1,\ldots,n\right\}$ and determined as follows; there is again a dependence on the parity of $p$.

In the case $p$ odd, the basis is computed as
\begin{equation}\label{basis2podd}
\begin{bmatrix}
N^p_{1,2}\\
N^p_{2,2}\\
\vdots\\
N^p_{n,2}\\
\end{bmatrix}= \begin{bmatrix}
\underbrace{\overbrace{\cdots\bigg\vert L^2_{n} \bigg\vert -L^2_{n}\bigg\vert L^2_{n}\bigg\vert}^{\frac{p+1}{2}}  I_{n}  \overbrace{ \bigg\vert  R^2_{n}   \bigg\vert-R^2_n\bigg\vert R^2_{n}    \bigg\vert    \cdots  }^{\frac{p+1}{2}}  }_{n+p+1}
\end{bmatrix}\begin{bmatrix}
N^p_{-p}\\
\vdots\\
N^p_{0}\\
N^p_{1}\\
\vdots\\
N^p_{n}\\
\end{bmatrix}, 
\end{equation}
where the set $\left\{ N_{i}^p,\ i=-p,\ldots,n  \right\}$
is defined by
\begin{equation}\label{bsplines2odd}
N_{i}^p(x)=\mathcal{N}_p\left(\frac{2n+1}{2}x-i\right),\quad i=-p,\ldots,n,
\end{equation}
and the matrices $L^2_{n}$ and $R^2_{n}$ are defined by
\begin{equation*}
L^2_{n}=\begin{bmatrix}
-I_{n} \bigg\vert  -J_{n}  \bigg\vert   O_{n}
\end{bmatrix},\quad R^2_{n}=\begin{bmatrix}
J_{n} \bigg\vert   O_{n}  \bigg\vert   -I_{n}
\end{bmatrix}.
\end{equation*}
In the case $p$ even, the basis is computed as
\begin{equation}\label{basis2peven}
\begin{bmatrix}
N^p_{1,2}\\
N^p_{2,2}\\
\vdots\\
N^p_{n,2}\\
\end{bmatrix}= \begin{bmatrix}
\underbrace{\overbrace{\cdots\bigg\vert L^2_{n} \bigg\vert -L^2_{n}\bigg\vert L^2_{n}\bigg\vert}^{\frac{p}{2}+1}  I_{n}  \overbrace{ \bigg\vert  R^2_{n}   \bigg\vert-R^2_n\bigg\vert R^2_{n}    \bigg\vert   \cdots  }^{\frac{p}{2}}  }_{n+p+1}
\end{bmatrix}\begin{bmatrix}
N^p_{-p}\\
\vdots\\
N^p_{0}\\
N^p_{1}\\
\vdots\\
N^p_{n}\\
\end{bmatrix},
\end{equation}
where the set $\left\{ N_{i}^p,\ i=-p,\ldots,n  \right\}$
is given by
\begin{equation}\label{bsplines2even}
N_{i}^p(x)=\mathcal{N}_p\left(\frac{2n+1}{2}x-\left(i-\frac{1}{2}\right)\right),\quad i=-p,\ldots,n. 
\end{equation}

\subsection{Structured matrices of Toeplitz-minus-Hankel or Toeplitz-plus-Hankel form}\label{sec:toeplitz+hankel}
We describe three known results regarding closed-form expressions for eigenvalues and eigenvectors of structured matrices of Toeplitz-minus-Hankel or Toeplitz-plus-Hankel form. Further details and other types of matrices can be found in \cite{deng2021analytical} (see also \cite[Table~1.2]{bolten2023note}).
All theorems presented here are taken from \cite{deng2021analytical}, adjusted to align more effectively with the notation of the GLT symbol presented in \cite{garoni2017generalized}. Additionally, various matrix algebras can be found in \cite{bini1983spectral,bozzo1995use,di1995matrix}, illustrating the connections between these fields.

Suppose a given integer number $n\in\mathbb{N}^{\star}$ and a given function $a$ belonging to the Lebesgue space $L^1([-\pi,\pi])$. The Toeplitz matrix of size $n$ generated by $a$ is defined by
\begin{equation*}
T_n(a)=\left[a_{i-j}\right]_{i,j=1}^n=\begin{bmatrix}
a_0&a_{-1}&a_{-2}&\cdots & \cdots& a_{-(n-1)} \\
a_1&\ddots&\ddots&\ddots & &\vdots \\
a_2&\ddots&\ddots&\ddots&\ddots &\vdots\\
\vdots&\ddots&\ddots&\ddots&\ddots&a_{-2}\\
\vdots& &\ddots&\ddots&\ddots&a_{-1}\\
a_{n-1}&\cdots&\cdots&a_2&a_1&a_{0}\\
\end{bmatrix},
\end{equation*}
where $a_k$ represents the $k$-th Fourier coefficient of  $a$, given by 
\begin{equation*}
a_k=\frac{1}{2\pi}\displaystyle\int_{-\pi}^{\pi}a(\theta)e^{-{\ii}k\theta}\;{\dd}\theta,\quad k\in\mathbb{Z},
\end{equation*}
with $\ii^2={-1}$.

Let $p$ be a positive integer such that $n \geq p+1$. Given the vector $\bm{\alpha} = (\alpha_0, \alpha_1, \ldots, \alpha_p) \in \mathbb{R}^{p+1}$, we define the symmetric Toeplitz matrix of size $n$ associated with $\bm{\alpha}$ as follows
\begin{equation}\label{Toplitzalpha}
T_n^{\bm{\alpha}}=\begin{bmatrix}
\alpha_0&\alpha_1&\cdots&\alpha_p& 0 & & & &\\
\alpha_1&\alpha_0&\alpha_1&\cdots&\alpha_p& & & &\\
\vdots&\alpha_1&\alpha_0&\alpha_1&\cdots&\alpha_p& &&\\
\alpha_p&\cdots&\alpha_1&\alpha_0&\alpha_1&\cdots&\alpha_p&&\\
0 &\ddots&&\ddots&\ddots&\ddots&&\ddots& 0\\
&& \alpha_p&\cdots&\alpha_1&\alpha_0&\alpha_1&\cdots&\alpha_p\\
& & &\alpha_p&\cdots&\alpha_1&\alpha_0&\alpha_1&\vdots\\
&&  & &\alpha_p&\cdots&\alpha_1&\alpha_0&\alpha_1\\
& & & & 0 &\alpha_p&\cdots&\alpha_1&\alpha_0\\
\end{bmatrix}.
\end{equation}
Through straightforward calculations, we can deduce the function that generates this matrix. More precisely, we have
\begin{equation*}
T_n^{\bm{\alpha}}=T_{n}(g^{\bm{\alpha}}_p),
\end{equation*}
where 
\begin{equation}\label{g}
g_p^{\bm{\alpha}}(\theta)=\alpha_0+2\displaystyle\sum_{k=1}^p \alpha_k\cos(k\theta),\quad \forall  \theta \in [0,\pi].
\end{equation}

Before delving into the results, it is essential to introduce two distinct types of Hankel matrices of size $n$ that are directly linked to the vector $\bm{\alpha}$:
\begin{equation}\label{hankelmatricesdif}
H_n^{\bm{\alpha},1}=\begin{bmatrix}
\alpha_1&\alpha_2&\alpha_3&\cdots& \alpha_p&0 & & \\
\alpha_2&\alpha_3&\cdots&\alpha_p& & & & \\
\alpha_3&\cdots&\alpha_p& & & & \\
\vdots&\udots&&&&&&\\
\alpha_p& & & & & & &0 \\
 0& & & & & & &\alpha_p\\
 && & &&&\udots&\vdots\\
 &&& & &\alpha_p&\cdots&\alpha_3\\ 
 &&& &\alpha_p&\cdots&\alpha_3&\alpha_2\\
& &0 & \alpha_p&\cdots&\alpha_3&\alpha_2&\alpha_1\\
\end{bmatrix},\quad  
H_n^{\bm{\alpha},2}=\begin{bmatrix}
\alpha_2&\alpha_3&\alpha_4&\cdots& \alpha_p&0 & & \\
\alpha_3&\alpha_4&\cdots&\alpha_p& & & & \\
\alpha_4&\cdots&\alpha_p& & & & \\
\vdots&\udots&&&&&&\\
\alpha_p& & & & & & &0 \\
 0& & & & & & &\alpha_p\\
 && & &&&\udots&\vdots\\
 &&& & &\alpha_p&\cdots&\alpha_4\\ 
 &&& &\alpha_p&\cdots&\alpha_4&\alpha_3\\
& &0 & \alpha_p&\cdots&\alpha_4&\alpha_3&\alpha_2\\
\end{bmatrix}.
\end{equation}
Note that the initial element of the Hankel matrix $H_n^{\bm{\alpha},1}$ corresponds to the second element of the Toeplitz matrix \eqref{Toplitzalpha}. Similarly, the Hankel matrix $H_n^{\bm{\alpha},2}$ begins with the third element of the same Toeplitz matrix.

Now we have the necessary background to introduce the following theorems; their proofs can be found in~\cite{deng2021analytical}.
\begin{theorem}\label{closed1}
Let $n,p\in\mathbb{N}^{\star}$ such that $n\geq p+1$. Then, the eigenpairs of the matrix $T_n^{\bm{\alpha}}-H_n^{\bm{\alpha},2}$ can be expressed as
\begin{equation*}
\lambda_j\left(T_n^{\bm{\alpha}}-H_n^{\bm{\alpha},2}\right)=g_p^{\bm{\alpha}}\left(\frac{j\pi}{n+1}\right),\quad j=1,\ldots, n,
\end{equation*}
and
\begin{equation*}
u_{i,j}\left(T_n^{\bm{\alpha}}-H_n^{\bm{\alpha},2}\right)=\sqrt{\frac{2}{n+1}}\sin\left(\frac{ij\pi}{n+1}\right),\quad i,j=1,\ldots, n,
\end{equation*}
where $g_p^{\bm{\alpha}}$ is given by \eqref{g}.
\end{theorem}

\begin{theorem}\label{closed3}
Let $n,p\in\mathbb{N}^{\star}$ such that $n\geq p+1$. Then, the eigenpairs of the matrix $T_n^{\bm{\alpha}}+H_n^{\bm{\alpha},1}$ can be expressed as
\begin{equation*}
\lambda_j\left(T_n^{\bm{\alpha}}+H_n^{\bm{\alpha},1}\right)=g_p^{\bm{\alpha}}\left(\frac{(j-1)\pi}{n}\right),\quad j=1,\ldots, n,
\end{equation*}
and
\begin{equation*}
u_{i,j}\left(T_n^{\bm{\alpha}}+H_n^{\bm{\alpha},1}\right)=\sqrt{\frac{2}{n}}\;c_j\cos\left(\frac{(j-1)\pi}{n}\left(i-\frac{1}{2}\right)\right),\quad c_j=\left\{\begin{array}{ll}
    \frac{1}{\sqrt{2}},& j=1, \\
    1,& \text{otherwise},
\end{array}
\right.\quad i,j=1,\ldots, n,
\end{equation*}
where $g_p^{\bm{\alpha}}$ is given by \eqref{g}.
\end{theorem}

\begin{theorem}\label{closed2}
Let $n,p\in\mathbb{N}^{\star}$ such that $n\geq p+1$. Then, the eigenpairs of the matrix $T_n^{\bm{\alpha}}-H_n^{\bm{\alpha},1}$ can be expressed as
\begin{equation*}
\lambda_j\left(T_n^{\bm{\alpha}}-H_n^{\bm{\alpha},1}\right)=g_p^{\bm{\alpha}}\left(\frac{j\pi}{n}\right),\quad j=1,\ldots, n,
\end{equation*}
and
\begin{equation*}
u_{i,j}\left(T_n^{\bm{\alpha}}-H_n^{\bm{\alpha},1}\right)=\sqrt{\frac{2}{n}}\;c_j\sin\left(\frac{j\pi}{n}\left(i-\frac{1}{2}\right)\right),\quad  c_j=\left\{\begin{array}{ll}
    \frac{1}{\sqrt{2}},& j=n, \\
    1,& \text{otherwise},
\end{array}
\right. \quad i,j=1,\ldots, n,
\end{equation*}
where $g_p^{\bm{\alpha}}$ is given by \eqref{g}.
\end{theorem}

\begin{remark}
Theorems~\ref{closed1},~\ref{closed3}, and~\ref{closed2} precisely mirror those in \cite[Theorems~2.1,~2.4, and~2.2]{deng2021analytical}, differing only in the normalization of the eigenvectors. Furthermore, an alternative version of Theorem~\ref{closed1} can be found in \cite[Proposition~2.2]{bini1983spectral}, wherein the eigenvalues are provided implicitly.
\end{remark}

\begin{remark}
From Appendix~\ref{sec:B}, it is clear that the matrices in Theorems~\ref{closed1}, \ref{closed3}, and~\ref{closed2} belong to certain matrix algebras, namely
\begin{equation*}
T_n^{\bm{\alpha}}-H_n^{\bm{\alpha},2}\in\tau_n(0,0),\quad
T_n^{\bm{\alpha}}+H_n^{\bm{\alpha},1}\in\tau_n(1,1),
\quad T_n^{\bm{\alpha}}-H_n^{\bm{\alpha},1}\in\tau_n(-1,-1).
\end{equation*}
\end{remark}

\section{Eigenpairs of the optimal spline discretization matrices}\label{sec:closedformoptspaces}
Our objective is to establish closed-form expressions for the eigenvalues and eigenvectors of the matrices $M^{p,\bi}_n$ and $K^{p,\bi}_n$, $\bi=0,1,2$, defined in \eqref{allmatrices} and arising from the discretization in the spaces $\mathbb{S}_{p,n,\bi}^{\opt}$ related to the boundary conditions \eqref{Dbc}--\eqref{Mbc}. For the sake of unification, we introduce a general matrix $X^{p,r,\bi}$ as follows.

\begin{definition}\label{difmatrixX}
Let $p,n\in \mathbb{N}^{\star}$ and $r\in\mathbb{N}$ such that $r \leq p$.
Let $B_{\bi}^p=\left\{N^{p}_{i,\bi},\ i=1,\ldots,n\right\}$ be the basis constructed in Sections~\ref{sec:DBCBasis}--\ref{sec:MBCBasis} for $\bi=0,1,2$, respectively.
We consider the matrix
$X^{p,r,\bi}=\left( X^{p,r,\bi}_{i,j}\right)_{1\leq i,j\leq n}$, whose elements are given by
\begin{equation*}
X^{p,r,\bi}_{i,j} = \displaystyle\int_0^1 (N^{p}_{i,\bi})^{(r)}(x) (N^{p}_{j,\bi})^{(r)}(x)\;{\dd}x,\quad  i,j=1,\ldots,n,
\end{equation*}
for $\bi=0,1,2$.
In particular, in the cases $r=0,1$, we have $X^{p,0,\bi}=M_n^{p,\bi}$ and $X^{p,1,\bi}=K_n^{p,\bi}$.
\end{definition}

In this section, by leveraging the properties \eqref{suppandregucb}--\eqref{innerprodcb} of cardinal B-splines, we show that the matrices in Definition~\ref{difmatrixX} exhibit a Toeplitz-minus-Hankel or Toeplitz-plus-Hankel structure, for all the values of $p$. Then, the results from Section~\ref{sec:toeplitz+hankel} lead to closed-form expressions for the eigenvalues and eigenvectors of our matrices of interest.
The section is divided into three subsections, each corresponding to one of the boundary conditions \eqref{Dbc}--\eqref{Mbc}. 
Since all three boundary cases can be treated using the same technique, we primarily detail the Dirichlet case. However, to enhance readability, we delegate the technical proofs to Appendix~\ref{sec:A}. As a side result, for the mixed boundary conditions \eqref{Mbc}, we identify a new Hankel structure, not covered in \cite{deng2021analytical}.

\subsection{The Dirichlet case}\label{sec:closedformdiri}
Let us consider the discretization of the eigenvalue problem \eqref{pd1} in the space $\mathbb{S}_{p,n,0}^{\opt}$ defined in \eqref{Sopti} and related to the boundary conditions \eqref{Dbc}.
We aim to show that the matrix $X^{p,r,0}$ in Definition~\ref{difmatrixX} has a Toeplitz-minus-Hankel form.
Before going to the main result, two essential lemmas are in order. The first lemma is a key ingredient in the proof and the second one contributes to its reduction.
\begin{lemma}\label{integcal}
Let $a,b,\rho\in\mathbb{R}$, $p\in\mathbb{N}^{\star}$, and $r\in\mathbb{N}$ such that $r \leq p$. Then,
\begin{equation*}
\displaystyle\int_a^b \mathcal{N}_p^{(r)}(t) \mathcal{N}_p^{(r)}(t+\rho) \;{\dd}t=\displaystyle\int_{p+1-b-\rho}^{p+1-a-\rho} \mathcal{N}_p^{(r)}(t) \mathcal{N}_p^{(r)}(t+\rho) \;{\dd}t.
\end{equation*}
\end{lemma}
\begin{proof}
Fix $a,b,\rho\in\mathbb{R}$. By the symmetry property of cardinal B-splines \eqref{symmetrycb} we have
\begin{equation*}
\mathcal{N}_p^{(r)}(t)=(-1)^r\mathcal{N}_p^{(r)}(p+1-t),
\end{equation*}
and 
\begin{equation*}
\mathcal{N}_p^{(r)}(t+\rho)=(-1)^r\mathcal{N}_p^{(r)}(p+1-t-\rho).
\end{equation*}
Integrating the product of these functions and using a change of variables ($t\leftarrow p+1-t-\rho$) gives the stated result.
\end{proof}

\begin{lemma}\label{anti-sy}
Let $p,n\in\mathbb{N}^{\star}$ and $r\in\mathbb{N}$ such that $r \leq p$. Then, the matrix $X^{p,r,0}$ is both symmetric and centrosymmetric.
\end{lemma}
\begin{proof} 
From Definition~\ref{difmatrixX} it is clear that $X^{p,r,0}$ is a symmetric matrix. Moreover, from the properties of cardinal B-splines \eqref{suppandregucb}--\eqref{innerprodcb} we deduce
\begin{equation*}
N^{p}_{i}(x)=N^{p}_{n+1-p_0-p-i}(1-x), \quad i=-p,\ldots,n+1-p_0, \quad p_0=p-2\left\lfloor \frac{p}{2}\right\rfloor, \quad x\in[0,1],
\end{equation*}
where $N^{p}_{i}$ is defined in \eqref{bsplines0odd} and \eqref{bsplines0even}.
Note that $p_0$ allows us to unify the cases of even and odd degree.
Then, from the constructions \eqref{basis0podd} and \eqref{basis0peven} it follows
\begin{equation*}
N^{p}_{i,0}(x)=N^{p}_{n+1-i,0}(1-x), \quad i=1,\ldots,n, \quad x\in[0,1],
\end{equation*}
which implies that $X^{p,r,0}$ is a centrosymmetric matrix.
\end{proof}

\begin{remark}\label{Diimport:centrosymalw}
The symmetry and centrosymmetry of the matrix $X^{p,r,0}$ hold without any restriction on the size $n$. Therefore, the insights given in \cite{cantoni1976eigenvalues} become particularly valuable for studying the eigenvectors, eigenvalues, and the development of fast solvers.
\end{remark}

In the next theorem, we establish that the matrix $X^{p,r,0}$ takes the form of $T_n^{\bm{\alpha}}-H_n^{\bm{\alpha},2}$ for a specific vector $\bm{\alpha}$, where $T_n^{\bm{\alpha}}$ and $H_n^{\bm{\alpha},2}$ are defined in \eqref{Toplitzalpha} and \eqref{hankelmatricesdif}, respectively. From Theorem~\ref{closed1}, it is evident that a closed-form expression for the eigenvalues and eigenvectors requires at least $n \geq p+1$. However, in our case, we need to assume  $n \geq C(p) \geq p+1$ and we aim to determine the smallest possible value for $C(p)$ that gives us the structure $T_n^{\bm{\alpha}}-H_n^{\bm{\alpha},2}$ without any perturbation. 

\begin{theorem}\label{closedformdiri}
Let $p,n\in\mathbb{N}^{\star}$ and $r\in\mathbb{N}$ such that $n\geq \max\left\{p+1,p+\left\lfloor\frac{p}{2}\right\rfloor-1\right\}$ and $r \leq p$. Then,
\begin{equation}\label{whatweneedtoproof0}
X^{p,r,0}=(n+1)^{2r-1}\left(T_n^{\bm{\alpha}}-H_n^{\bm{\alpha},2}\right),
\end{equation}
where the elements of the vector $\bm{\alpha}$ are given by
\begin{equation}\label{alphaparam}
\alpha_k=(-1)^r\mathcal{N}_{2p+1}^{(2r)}(p+1-k),\quad k=0,\ldots,p.
\end{equation}
\end{theorem}

\begin{proof}
For ease of reading, the detailed proof is moved to Section~\ref{sec:A1} (see Appendix~\ref{sec:A}) and we only sketch the idea here. Assuming $n \geq p + 1$, we start by providing a simplified  definition of the basis $B_0^p$, derived from \eqref{basis0podd} and \eqref{basis0peven}. This allows us to obtain explicit expressions for the elements of the matrix $X^{p,r,0}$ by leveraging the properties of cardinal B-splines \eqref{suppandregucb}--\eqref{innerprodcb} and utilizing Lemma~\ref{integcal}. We do this in two steps.
In the first step, we show that the central part of the matrix corresponds to a Toeplitz form, 
\begin{equation*}
X^{p,r,0}_{i,j}=(n+1)^{2r-1}(-1)^{r}\mathcal{N}_{2p+1}^{(2r)}(p+1+i-j),\quad i,j=\left\lfloor\frac{p}{2}\right\rfloor+1,\ldots,n-\left\lfloor\frac{p}{2}\right\rfloor.
\end{equation*}
In the second step, under the stated conditions on $n$ and $p$, we show that the first block of size $\left(\left\lfloor\frac{p}{2}\right\rfloor,n\right)$ corresponds to a Toeplitz-minus-Hankel form,
\begin{equation*}
X^{p,r,0}_{i,j}=(n+1)^{2r-1}\left((-1)^r\mathcal{N}_{2p+1}^{(2r)}(p+1+i-j)-(-1)^r\mathcal{N}_{2p+1}^{(2r)}(p+1+i+j) \right),\quad i=1,\ldots,\left\lfloor\frac{p}{2}\right\rfloor,\quad j=i,\ldots,n.
\end{equation*}
Exploiting the symmetry and centrosymmetry of $X^{p,r,0}$ (see Lemma~\ref{anti-sy}) completes the proof.
\end{proof}
\begin{remark}\label{generalizationX}
As a generalization of $X^{p,r,0}$, one can consider two indices $r_1, r_2\in\mathbb{N}$ instead of just $r$ in Definition~\ref{difmatrixX}. Then, the same analysis can be followed to derive a more diverse range of Toeplitz-minus/plus-Hankel matrices that depend on the pair $(r_1, r_2)$. For brevity, we omit this analysis.
\end{remark}

The next corollary follows directly from Theorems~\ref{closed1} and~\ref{closedformdiri}, taking into account that the function $g_p^r$ in \eqref{g_r} is equal to $g_p^{\bm{\alpha}}$ in \eqref{g}, with the elements of $\bm{\alpha}$ specified in \eqref{alphaparam}.
\begin{corollary}\label{corollaryeigdiri}
Let $p,n\in\mathbb{N}^{\star}$ and $r\in\mathbb{N}$ such that $n\geq \max\left\{p+1,p+\left\lfloor\frac{p}{2}\right\rfloor-1\right\}$ and $r\leq p$. Then, the eigenpairs of the matrix $X^{p,r,0}$ can be expressed as
\begin{equation*}
\lambda_j\left(X^{p,r,0}\right)=(n+1)^{2r-1}\;g^{r}_p\left(\frac{j\pi}{n+1}\right),\quad j=1,\ldots,n,
\end{equation*}
and
\begin{equation*}
u_{i,j}\left(X^{p,r,0}\right)=\sqrt{\frac{2}{n+1}}\sin\left(\frac{ij\pi}{n+1}\right),\quad i,j=1,\ldots, n,
\end{equation*}
where $g_p^r$ is given by
\begin{equation}\label{g_r}
g_p^r(\theta)=
 (-1)^{r}\mathcal{N}_{2p+1}^{(2r)}(p+1)+2(-1)^{r}\displaystyle\sum_{k=1}^p\mathcal{N}^{(2r)}_{2p+1}(p+1-k)\cos(k\theta),\quad \theta\in[0,\pi], \quad r\leq p.\\
\end{equation} 
\end{corollary}

We now state a corollary regarding the closed form of eigenvalues and eigenvectors of the discretized Laplace operator \eqref{systfindimi} with the Dirichlet boundary conditions \eqref{Dbc}. We recall that the matrices $M_n^{p,0}$ and $K_n^{p,0}$ defined in \eqref{allmatrices} represent a particular case of $X^{p,r,0}$ ($r=0,1$). 
\begin{corollary}\label{coroloryclosedform0}
Let $p,n\in\mathbb{N}^{\ast}$ such that $n\geq \max\left\{p+1,p+\left\lfloor\frac{p}{2}\right\rfloor-1\right\}$. Then, the eigenpairs of the matrix $L_n^{p,0}=\left(M_n^{p,0}\right)^{-1}K_n^{p,0}$ can be expressed as
\begin{equation*}
\lambda_j\left(L_n^{p,0}\right)=(n+1)^{2}\;\dfrac{g^{1}_p\left(\frac{j\pi}{n+1}\right)}{g^{0}_p\left(\frac{j\pi}{n+1}\right)},\quad j=1,\ldots,n,
\end{equation*}
and
\begin{equation*}
u_{i,j}\left(L_n^{p,0}\right)=\sqrt{\frac{2}{n+1}}\sin\left(\frac{ij\pi}{n+1}\right),\quad  i,j=1,\ldots, n,
\end{equation*}
where $g_p^{r}$ is given by \eqref{g_r} for $r=0,1$.
\end{corollary}
\begin{proof}
The results follow directly from Corollary \ref{corollaryeigdiri} by considering the cases $r=0$ and $r=1$. Additionally, it was shown in \cite{garoni2014spectrum} that
\begin{equation*}
\left(\frac{4}{\pi^2}\right)^{p+1}\leq g^0_p(\theta)\leq 1,\quad \theta \in[0,\pi],
\end{equation*}
ensuring the validity of the division by $g^0_p\left(\frac{j\pi}{n+1}\right)$, $j=1,\ldots,n$.
\end{proof}

\begin{remark}\label{rmk:symbols}
The functions $g^{1}_p$ and $g^{0}_p$ are the so-called (GLT) symbols of the stiffness and mass matrices resulting from the Galerkin discretization in the full spline space $\mathbb{S}_{p,\bm{\tau}}$, where $\bm{\tau}$ corresponds to a uniform grid.
These symbols are independent of the boundary conditions.
There is an extensive literature on their construction and on the analysis of their properties; see, e.g., \cite{donatelli2016spectral,ekstrom2018eigenvalues,garoni2014spectrum} and references therein.
The above results also show that $g^{1}_p$ and $g^{0}_p$ are the (GLT) symbols of the same matrices when the discretization is performed in  optimal spline subspaces (or in reduced spline spaces; see Section~\ref{sec:reducedspaces}) of the same degree.
\end{remark}

\begin{remark}\label{rmk:outlierfree0}
The explicit closed-form expressions, described in Corollary~\ref{coroloryclosedform0}, for the approximated spectrum of the Laplace operator with Dirichlet boundary conditions give an algebraic confirmation that the optimal spline subspaces $\mathbb{S}_{p,n,0}^{\opt}$ lead to outlier-free discretizations. Indeed, it was shown in \cite[Theorem~1]{ekstrom2018eigenvalues} that a uniform sampling of the function
\begin{equation*}
e_p(\theta)= \dfrac{g^{1}_p(\theta)}{g^{0}_p(\theta)},\ \ \theta \in [0,\pi]
\end{equation*}
provides an approximation of the foreseen spectrum, with no outliers; see also \cite{garoni2019symbol}.
In particular, by \cite[Lemma~1]{ekstrom2018eigenvalues} we have for $p\geq 1$ and $\theta\in [0,\pi]$,
\begin{equation}\label{error-bound}
0\leq\dfrac{e_p(\theta)-\theta^2}{\theta^2}\leq \dfrac{4\pi(\pi-\theta)}{(2\pi-\theta)^2}\left(\dfrac{\theta}{2\pi-\theta}\right)^{2p}+5\left(\dfrac{\theta}{2\pi+\theta}\right)^{2p}.
\end{equation}
By Corollary~\ref{coroloryclosedform0} and setting $\theta_j= \frac{j\pi}{n+1}$, $j=1, \ldots, n$,
we get
\begin{equation*}
\dfrac{\lambda_j\left(L_n^{p,0}\right)-(j\pi)^2}{(j\pi)^2}=
\dfrac{(n+1)^2}{(j\pi)^2}(e_p(\theta_j)-(\theta_j)^2)=\dfrac{e_p(\theta_j)-(\theta_j)^2}{(\theta_j)^2}.
\end{equation*}
Thus, for all $j=1,\ldots,n$, the inequalities in \eqref{error-bound} show that $\lambda_j\left(L_n^{p,0}\right)$ is an approximation of the $j$-th eigenvalue of the Laplace operator with Dirichlet boundary conditions, see \eqref{exactDbc}, whose relative error converges to zero as $p$ increases. Thus, the approximation is outlier-free.
\end{remark}

\subsection{The Neumann case}\label{sec:closedformneum}
After addressing the Dirichlet boundary conditions \eqref{Dbc} in Section~\ref{sec:closedformdiri}, our attention shifts to the Neumann boundary conditions \eqref{Nbc}. Although the proof of the main result and the analysis in this section is similar to the Dirichlet case, we still include it with the necessary details in Appendix~\ref{sec:A} for two reasons. Firstly, this serves as an illustration how to handle cases not explicitly covered in this work (see Remark~\ref{generalizationX}). Secondly, and more importantly, the analysis of the reduced spline space proposed in Section~\ref{sec:reducedspaces} can be viewed as a small variation of the Neumann scenario. Consequently, the results obtained here will straightforwardly imply the results of the reduced spline space.

Let us consider the discretization of the eigenvalue problem \eqref{pd1} in the space $\mathbb{S}_{p,n,1}^{\opt}$ defined in \eqref{Sopti} and related to the boundary conditions \eqref{Nbc}. The matrix $X^{p,r,1}$ in Definition~\ref{difmatrixX} exhibits the structure of $T_n^{\bm{\alpha}}+H_n^{\bm{\alpha},1}$ for a specific vector $\bm{\alpha}$, where $T_n^{\bm{\alpha}}$ and $H_n^{\bm{\alpha},1}$ are given by \eqref{Toplitzalpha} and \eqref{hankelmatricesdif}, respectively. 
Following a proof akin to that of Lemma~\ref{anti-sy}, it can be shown that $X^{p,r,1}$ is centrosymmetric. This property will be exploited in the proof of the next theorem, which shows the Toeplitz-plus-Hankel form of $X^{p,r,1}$.

\begin{theorem}\label{closedformneum}
Let $p,n\in\mathbb{N}^{\star}$ and $r\in\mathbb{N}$ such that $n\geq \max\left\{2p-\left\lfloor \frac{p}{2}\right\rfloor,2p-2\left\lfloor\frac{p}{2}\right\rfloor+1\right\}$ and $r\leq p$. Then,
\begin{equation}\label{whatweneedtoproof1}
X^{p,r,1}=n^{2r-1}\left(T_n^{\bm{\alpha}}+H_n^{\bm{\alpha},1}\right),
\end{equation}
where the elements of the vector $\bm{\alpha}$ are given by \eqref{alphaparam}.
\end{theorem}
\begin{proof}
A similar proof strategy can be employed to the one for Theorem~\ref{closedformdiri} in the Dirichlet case. The detailed steps can be found in Section~\ref{sec:A2} (see Appendix~\ref{sec:A}).
\end{proof}

By combining Theorems~\ref{closed3} and~\ref{closedformneum} we directly obtain a closed-form expression for the eigenvalues and eigenvectors of the matrix $X^{p,r,1}$.
\begin{corollary}\label{corollaryeigneum}
Let $p,n\in\mathbb{N}^{\star}$ and $r\in\mathbb{N}$ such that $n\geq \max\left\{2p-\left\lfloor \frac{p}{2}\right\rfloor,2p-2\left\lfloor\frac{p}{2}\right\rfloor+1\right\}$ and $r\leq p$. Then, the eigenpairs of the matrix $X^{p,r,1}$ can be expressed as
\begin{equation*}
\lambda_j\left(X^{p,r,1}\right)=n^{2r-1}\;g^{r}_p\left(\frac{(j-1)\pi}{n}\right),\quad j=1,\ldots,n,
\end{equation*}
and
\begin{equation*}
u_{i,j}\left(X^{p,r,1}\right)=\sqrt{\frac{2}{n}}\;c_j\cos\left(\frac{(j-1)\pi}{n}\left(i-\frac{1}{2}\right)\right),\quad c_j=\left\{\begin{array}{ll}
    \frac{1}{\sqrt{2}},& j=1, \\
    1,& \text{otherwise},
\end{array}\quad i,j=1,\ldots, n,
\right.
\end{equation*}
where $g_p^{r}$ is given by \eqref{g_r}.
\end{corollary}

Thanks to Corollary~\ref{corollaryeigneum} and the definition of $X^{p,r,1}$ ($r=0,1$), we conclude with the following result on the eigenpairs of the system \eqref{systfindimi} related to the basis $B_1^p$.
\begin{corollary}\label{coroloryclosedform1}
Let $p,n\in\mathbb{N}^{\ast}$ such that $n\geq \max\left\{2p-\left\lfloor \frac{p}{2}\right\rfloor,2p-2\left\lfloor\frac{p}{2}\right\rfloor+1\right\}$. Then, the eigenpairs of the matrix $L_n^{p,1}=\left(M_n^{p,1}\right)^{-1}K_n^{p,1}$ can be expressed as
\begin{equation*}
\lambda_j\left(L_n^{p,1}\right)=n^{2}\;\frac{g^{1}_p\left(\frac{(j-1)\pi}{n}\right)}{g^{0}_p\left(\frac{(j-1)\pi}{n}\right)},\quad j=1,\ldots,n,
\end{equation*}
and
\begin{equation*}
u_{i,j}\left(L_n^{p,1}\right)=\sqrt{\frac{2}{n}}\;c_j\cos\left(\frac{(j-1)\pi}{n}\left(i-\frac{1}{2}\right)\right),\quad c_j=\left\{\begin{array}{ll}
    \frac{1}{\sqrt{2}},& j=1, \\
    1,& \text{otherwise},
\end{array}\quad i,j=1,\ldots, n,
\right.
\end{equation*}
where $g_p^{r}$ is given by \eqref{g_r} for $r=0,1$.
\end{corollary}
\begin{proof}
Since the functions $g_p^0$ and $g_p^1$ are also given by \eqref{g_r}, the argument of the proof is the same as that of Corollary~\ref{coroloryclosedform0}.
\end{proof}

\begin{remark}\label{rmk:outlierfree1}
Corollary~\ref{coroloryclosedform1} gives an algebraic confirmation that
the optimal spline subspaces $\mathbb{S}_{p,n,1}^{\opt}$ lead to outlier-free discretizations for the spectrum of the Laplace operator with Neumann boundary conditions.
This follows from \eqref{exactNbc} in the same way as described in Remark~\ref{rmk:outlierfree0}. In particular, by setting $\theta_j= \frac{(j-1)\pi}{n}$, $j=2, \ldots, n$,
we get
\begin{equation*}
\dfrac{\lambda_j\left(L_n^{p,1}\right)-((j-1)\pi)^2}{((j-1)\pi)^2}=
\dfrac{n^2}{((j-1)\pi)^2}(e_p(\theta_j)-(\theta_j)^2)=\dfrac{e_p(\theta_j)-(\theta_j)^2}{(\theta_j)^2}.
\end{equation*}
Moreover, $\lambda_1\left(L_n^{p,1}\right)$ agrees with the exact value of the first eigenvalue in \eqref{exactNbc} because $g_p^1(0)=0$.
\end{remark}

\subsection{The mixed case}\label{sec:closedformmixed}
The mixed boundary conditions \eqref{Mbc} represent a combination of a Dirichlet boundary condition and a Neumann boundary condition. Since, by nature, these boundary conditions are not symmetric, the mass and stiffness matrices arising from this case are not centrosymmetric. However, these matrices can be described by combining the matrices of the previously studied boundaries and this gives rise to a new class of matrices. We also establish closed-form expressions for the eigenvalues and eigenvectors of this type of matrices. The same reasoning can be applied if we permute the mixed boundary conditions.

We define a new Hankel matrix, which combines the Hankel matrices of the Dirichlet case ($H_n^{\bm{\alpha},2}$) and the Neumann case ($H_n^{\bm{\alpha},1}$):
\begin{equation*}
H_n^{\bm{\alpha},2,1}=\begin{bmatrix}
-\alpha_2&-\alpha_3&-\alpha_4&\cdots& -\alpha_p&0 & & \\
-\alpha_3&-\alpha_4&\cdots&-\alpha_p& & & & \\
-\alpha_4&\cdots&-\alpha_p& & & & \\
\vdots&\udots&&&&&&\\
-\alpha_p& & & & & & & 0\\
 0& & & & & & &\alpha_p\\
 && & &&&\udots&\vdots\\
 &&& & &\alpha_p&\cdots&\alpha_3\\ 
 &&& &\alpha_p&\cdots&\alpha_3&\alpha_2\\
& &0 & \alpha_p&\cdots&\alpha_3&\alpha_2&\alpha_1\\
\end{bmatrix}.
\end{equation*}
The next theorem is the key to the closed-form expression of the eigenpairs considered in this section. It can be proved using the same technique as done in \cite[Theorem~2.1]{deng2021analytical}; for brevity, we omit the details.
\begin{theorem}\label{closedformpropmixd}
Let $n,p\in\mathbb{N}^{\star}$ such that $n\geq p+1$. Then, the eigenpairs of the matrix $T_n^{\bm{\alpha}}+H_n^{\bm{\alpha},2,1}$ can be expressed as 
\begin{equation*}
\lambda_j\left(T_n^{\bm{\alpha}}+H_n^{\bm{\alpha},2,1}\right)=g_p^{\bm{\alpha}}\left(\frac{(2j-1)\pi}{2n+1}\right),\quad j=1,\ldots, n,
\end{equation*}
and
\begin{equation*}
u_{i,j}\left(T_n^{\bm{\alpha}}+H_n^{\bm{\alpha},2,1}\right)=\sqrt{\frac{4}{2n+1}}\sin\left(\frac{i(2j-1)\pi}{2n+1}\right),\quad i,j=1,\ldots, n,
\end{equation*}
where $g_p^{\bm{\alpha}}$ is given by \eqref{g}.   
\end{theorem}
\begin{remark}
From Appendix~\ref{sec:B} it is clear that the matrix $T_n^{\bm{\alpha}}+H_n^{\bm{\alpha},2,1}$ falls within the matrix algebra $\tau_n(0,1)$. This algebra is  referred to as the algebra $\Gamma$ in \cite{di1995matrix} and initially introduced in \cite{bapat1991hypergroups}.
\end{remark}

Now, with all the necessary prerequisites, we present the main result on the Toeplitz-plus-Hankel form of the matrix $X^{p,r,2}$ in Definition~\ref{difmatrixX}.
\begin{theorem}\label{closedformmixed}
Let $p,n\in\mathbb{N}^{\star}$ and $r\in\mathbb{N}$ such that $n\geq \max\left\{p+1,p+\left\lfloor\frac{p}{2}\right\rfloor\right\}$ and $r\leq p$. Then,
\begin{equation*}
X^{p,r,2}=\left(\frac{2n+1}{2}\right)^{2r-1}\left(T_n^{\bm{\alpha}}+H_n^{\bm{\alpha},2,1}\right),
\end{equation*}
where the elements of the vector $\bm{\alpha}$ are given by \eqref{alphaparam}.
\end{theorem}
\begin{proof}
We can again employ a similar proof strategy as in the previous boundary cases (Theorems~\ref{closedformdiri} and~\ref{closedformneum}).
However, due to the lack of centrosymmetry of the matrix $X^{p,r,2}$, we need to do some extra work in our proof. More precisely, we divide the second step into two parts: the first part is akin to the Dirichlet case, while the second part resembles the Neumann case. The detailed steps are provided in Section~\ref{sec:A3} (see Appendix~\ref{sec:A}). 
\end{proof}

Since the structure considered in Theorem~\ref{closedformmixed} is a specific case of Theorem~\ref{closedformpropmixd}, we directly arrive at the following result.
\begin{corollary}\label{corollaryeigmixed}
Let $p,n\in\mathbb{N}^{\star}$ and $r\in\mathbb{N}$ such that $n\geq \max\left\{p+1,p+\left\lfloor\frac{p}{2}\right\rfloor\right\}$ and $r\leq p$. Then, the eigenpairs of the matrix $X^{p,r,2}$ can be expressed as
\begin{equation*}
\lambda_j\left(X^{p,r,2}\right)=\left(\frac{2n+1}{2}\right)^{2r-1}g_p^{r}\left(\frac{(2j-1)\pi}{2n+1}\right),\quad j=1,\ldots, n,
\end{equation*}
and
\begin{equation*}
u_{i,j}\left(X^{p,r,2}\right)=\sqrt{\frac{4}{2n+1}}\sin\left(\frac{i(2j-1)\pi}{2n+1}\right),\quad i,j=1,\ldots, n,
\end{equation*}
where $g_p^{r}$ is given by \eqref{g_r}.
\end{corollary}
The closed form of the eigenpairs of the system \eqref{systfindimi} related to the basis $B_2^p$ follows from the above corollary, taking $r=0,1$.
\begin{corollary}\label{coroloryclosedform2}
Let $p,n\in\mathbb{N}^{\ast}$ such that $n\geq \max\left\{p+\left\lfloor\frac{p}{2}\right\rfloor,p+1\right\}$. Then, the eigenpairs of the matrix
$L_n^{p,2}=\left(M_n^{p,2}\right)^{-1}K_n^{p,2}$ can be expressed as
\begin{equation*}
\lambda_j\left(L_n^{p,2}\right)=\left(\frac{2n+1}{2}\right)^{2}\frac{g_p^{1}\left(\frac{(2j-1)\pi}{2n+1}\right)}{g_p^{0}\left(\frac{(2j-1)\pi}{2n+1}\right)},\quad j=1,\ldots, n,
\end{equation*}
and
\begin{equation*}
u_{i,j}\left(L_n^{p,2}\right)=\sqrt{\frac{4}{2n+1}}\sin\left(\frac{i(2j-1)\pi}{2n+1}\right),\quad  i,j=1,\ldots, n,
\end{equation*}
where $g_p^{r}$ is given by \eqref{g_r} for $r=0,1$.
\end{corollary}

\begin{remark}\label{rmk:outlierfree2}
Corollary~\ref{coroloryclosedform2} gives an algebraic confirmation that
the optimal spline subspaces $\mathbb{S}_{p,n,2}^{\opt}$ lead to outlier-free discretizations for the spectrum of the Laplace operator with mixed boundary conditions.
This follows from \eqref{exactMbc} in the same way as described in Remark~\ref{rmk:outlierfree0}. In particular, by setting $\theta_j= \frac{(2j-1)\pi}{2n+1}$, $j=1, \ldots, n$,
we get
\begin{equation*}
\dfrac{4\lambda_j\left(L_n^{p,1}\right)-((2j-1)\pi)^2}{((2j-1)\pi)^2}=
\dfrac{(2n+1)^2}{((2j-1)\pi)^2}(e_p(\theta_j)-(\theta_j)^2)=\dfrac{e_p(\theta_j)-(\theta_j)^2}{(\theta_j)^2}.
\end{equation*}
\end{remark}

\section{Other outlier-free spline spaces}\label{sec:otherspaces}
Upon completing the derivation of closed-form expressions for the eigenpairs of the discretization matrices in optimal spline subspaces, 
we now aim to illustrate that a similar analysis can be seamlessly applied to other (univariate) reduced spline spaces, known in the outlier-free IGA literature. This reinforces the flexibility and applicability of our methodology for addressing various scenarios. Furthermore, we expand all our findings from the univariate to the multivariate tensor-product setting. 
For simplicity of presentation, in this section, we concentrate exclusively on Dirichlet boundary conditions but similar results hold for other boundary conditions as well.

\subsection{Reduced spline spaces}\label{sec:reducedspaces}
Here, we address the reduced spline spaces considered in \cite{hiemstra2021removal} (see also \cite{manni2022application,sogn2019robust,takacs2016approximation}). We start by defining these spaces in the Dirichlet case \eqref{Dbc} and we recall the B-spline-like basis from \cite{manni2022application}. Subsequently, we present the main result, which directly stems from the analysis provided for the optimal spline subspace addressed in Section~\ref{sec:closedformneum}.

Let us approximate the eigenvalue problem \eqref{pd1} under the Dirichlet boundary conditions \eqref{Dbc} using the reduced spline space $\overline{\mathbb{S}}_{p,n,0}$, given by 
\begin{equation*}
\overline{\mathbb{S}}_{p,n,0}=\mathbb{S}_{p,\overline{\bm{\tau}},0}^{p-1},
\end{equation*}
where $\mathbb{S}_{p,\overline{\bm{\tau}},0}^{p-1}$ is defined in \eqref{S0l}, taking $\ell=p-1$, and $\overline{\bm{\tau}}$ corresponds to the uniform grid
\begin{equation*}
\overline{\bm{\tau}}=\left(0,\frac{1}{n},\frac{2}{n},\ldots,\frac{n-1}{n},1 \right).
\end{equation*}
It has been observed numerically in \cite{hiemstra2021removal} that these spaces are outlier-free. Here we confirm these observations theoretically and actually provide closed-form expressions for the discrete eigenvalues and eigenvectors of the matrices arising from such discretizations.

We first remark that $\mathbb{S}_{p,\overline{\bm{\tau}},0}^{p}=\mathbb{S}_{p,\overline{\bm{\tau}},0}^{p-1}$ and $\overline{\mathbb{S}}_{p,n+1,0}=\mathbb{S}_{p,n,0}^{\opt}$ for $p$ odd;
see also \cite[Remark~4.1]{manni2022application}.
Therefore, the results from Section~\ref{sec:closedformdiri} are valid here for $p$ odd and  it suffices to restrict the analysis to the case of $p$ even. Hereafter, we assume that $p$ is even in this section, and we have $\dim\left(\overline{\mathbb{S}}_{p,n,0}\right)=n$.
From \cite{manni2022application}, the B-spline-like basis $\overline{B}_0^p=\left\{\overline{N}^{p}_{i,0},\ i=1,\ldots,n\right\}$ is computed as
\begin{equation}\label{basisrpeven}
\begin{bmatrix}
\overline{N}^p_{1,0}\\
\overline{N}^p_{2,0}\\
\vdots\\
\overline{N}^p_{n,0}\\
\end{bmatrix}= \begin{bmatrix}
\underbrace{\overbrace{\cdots\bigg\vert \overline{L}_{n} \bigg\vert \overline{L}_{n}\bigg\vert}^{\frac{p}{2}}  I_{n}  \overbrace{ \bigg\vert  \overline{R}_{n}  \bigg\vert \overline{R}_{n}     \bigg\vert   \cdots  }^{\frac{p}{2}}  }_{n+p}
\end{bmatrix}\begin{bmatrix}
N^p_{-p}\\
\vdots\\
N^p_{0}\\
N^p_{1}\\
\vdots\\
N^p_{n-1}\\
\end{bmatrix},
\end{equation}
where the set $\left\{ N_{i}^p,\ i=-p,\ldots,n-1  \right\}$
is derived from the cardinal B-spline $\mathcal{N}_p$ in \eqref{eq:cardinal-B-spline-2} as
\begin{equation*}
N_{i}^p(x)=\mathcal{N}_p\left(nx-i\right),\quad i=-p,\ldots,n-1,
\end{equation*}
and the matrices $ \overline{L}_{n}$ and $\overline{R}_{n}$ are defined by
\begin{equation*}
\overline{L}_{n}=\begin{bmatrix}
I_{n}   \bigg\vert   -J_{n}
\end{bmatrix},\quad \overline{R}_{n}=\begin{bmatrix}
-J_{n}  \bigg\vert   I_{n}
\end{bmatrix}.
\end{equation*}

Let $\overline{M}_n^{p,0}$ and $\overline{K}_n^{p,0}$ denote the mass matrix and the stiffness matrix, respectively, related to the basis $\overline{B}_0^p$.
For the sake of unification, similar to Definition~\ref{difmatrixX}, we introduce a general matrix $\overline{X}^{p,r,0}$.
\begin{definition}\label{difmatrixreduced}
Let $p,n\in \mathbb{N}^{\star}$ and $r\in\mathbb{N}$ such that $p$ even and $r \leq p$.
Let $\overline{B}_0^p=\left\{\overline{N}^{p}_{i,0},\ i=1,\ldots,n\right\}$ be the basis constructed in \eqref{basisrpeven}.
We consider the matrix
$\overline{X}^{p,r,0}=\left( \overline{X}^{p,r,0}_{i,j}\right)_{1\leq i,j\leq n}$, whose elements are given by
\begin{equation*}
\overline{X}^{p,r,0}_{i,j} = \displaystyle\int_0^1 (\overline{N}^{p}_{i,0})^{(r)}(x)(\overline{N}^{p}_{j,0})^{(r)}(x)\;{\dd}x,\quad  i,j=1,\ldots,n.
\end{equation*}
In particular, in the cases $r=0,1$, we have $\overline{X}^{p,r,0}=\overline{M}_n^{p,0}$ and $\overline{X}^{p,r,0}=\overline{K}_n^{p,0}$.
\end{definition}

Drawing upon the proof of Lemma~\ref{anti-sy}, one can check that the matrix $\overline{X}^{p,r,0}$ is centrosymmetric without any restriction on the size $n$.
In the next theorem, we establish that $\overline{X}^{p,r,0}$ takes the form of $T_n^{\bm{\alpha}}-H_n^{\bm{\alpha},1}$, where $T_n^{\bm{\alpha}}$  and $H_n^{\bm{\alpha},1}$  are defined by \eqref{Toplitzalpha} and \eqref{hankelmatricesdif}, respectively. Then, by applying Theorem~\ref{closed2}, we obtain a closed-form expression for the eigenvalues and eigenfunctions.
\begin{theorem}\label{closedformreduced}
Let $p,n\in\mathbb{N}^{\star}$ and $r\in\mathbb{N}$ such that $p$ even, $n\geq p+\frac{p}{2}$, and $r\leq p$. Then,
\begin{equation}\label{redfinwhat}
\overline{X}^{p,r,0}=n^{2r-1}\left(T_n^{\bm{\alpha}}-H_n^{\bm{\alpha},1}\right),
\end{equation}
where the elements of the vector $\bm{\alpha}$ are given by \eqref{alphaparam}.
\end{theorem}
\begin{proof}
Using the assumption $n\geq p+1$, the expression \eqref{basisrpeven} simplifies to 
\begin{equation*}
\begin{bmatrix}
\overline{N}^p_{1,0}\\
\overline{N}^p_{2,0}\\
\vdots\\
\overline{N}^p_{n,0}\\
\end{bmatrix} = 
\begin{bmatrix}
\underbrace{\overbrace{\; -J_{n} \;\bigg\vert}^{ \frac{p}{2}}  \; I_{n}  \; \overbrace{   \bigg\vert \; -J_{n}  \;   }^{ \frac{p}{2}}  }_{n+p}
\end{bmatrix}
\begin{bmatrix}
N^p_{-p}\\
\vdots\\
N^p_{0}\\
N^p_{1}\\
\vdots\\
N^p_{n-1}\\
\end{bmatrix},
\end{equation*}
and we obtain
\begin{equation}\label{difexplbasisrd}
\overline{N}^p_{i,0}=\left\{
\begin{array}{ll}
 -N^p_{q-i}+N^p_{q+i-1},& i\in\left\{1,\ldots,-q\right\},\\[0.1cm]
 N^p_{q+i-1},& i\in\left\{-q+1,\ldots,n+q\right\},\\[0.1cm]
N^p_{q+i-1}-N^p_{q+2n-i},& i\in\left\{n+q+1,\ldots,n\right\},
\end{array}
 \right.
\end{equation}
where 
\begin{equation*}
q=\left\lfloor\frac{p}{2}\right\rfloor-p=-\frac{p}{2}.
\end{equation*}

Taking into account that $p$ is even and $\left\lfloor\frac{p}{2}\right\rfloor=\frac{p}{2}=-q$, we observe a strong similarity between \eqref{difexplbasisrd} and \eqref{difexplbasis1} in the proof of Theorem~\ref{sec:closedformneum} (see Section~\ref{sec:A2}).
This similarity implies that we can follow exactly the same proof, with the only difference that for the reduced spline space we need to replace the general form of the basis functions in \eqref{generalformbasis1} with
\begin{equation*}
\overline{N}_{i,0}^p=N^p_{q+i-1}-N^p_{q-k_i}.
\end{equation*}
Therefore, the general form of the elements of $\overline{X}^{p,r,0}$ is obtained by replacing the sign `$+$' in the two terms of \eqref{exprelmX1} that have only $k_i$ or $k_j$ by `$-$'.
Then, we can simply follow the same steps as in the proof of Theorem~\ref{sec:closedformneum}, under the assumption $n\geq p+\frac{p}{2}$, and we arrive at the identity \eqref{redfinwhat}.
\end{proof}

By combining Theorems~\ref{closed2} and~\ref{closedformreduced} we directly deduce a closed-form expression for the eigenvalues and eigenvectors of the matrix $\overline{X}^{p,r,0}$.
\begin{corollary}\label{corollaryeigreduced}
Let $p,n\in\mathbb{N}^{\star}$ and $r\in\mathbb{N}$ such that $p$ even, $n\geq p+\frac{p}{2}$, and $r\leq p$. Then, the eigenpairs of the matrix $\overline{X}^{p,r,0}$ can be expressed as 
\begin{equation*}
\lambda_j\left(\overline{X}^{p,r,0}\right)=n^{2r-1}g_p^{r}\left(\frac{j\pi}{n}\right),\quad j=1,\ldots, n,
\end{equation*}
and
\begin{equation*}
u_{i,j}\left(\overline{X}^{p,r,0}\right)=\sqrt{\frac{2}{n}}\;c_j\sin\left(\frac{j\pi}{n}\left(i-\frac{1}{2}\right)\right),\quad   c_j=\left\{\begin{array}{ll}
    \frac{1}{\sqrt{2}},& j=n, \\
    1,& \text{otherwise},
\end{array}
\right.\quad  i,j=1,\ldots, n,
\end{equation*}
where $g_p^{r}$ is given by \eqref{g_r}. 
\end{corollary}

A simplified version of Corollary~\ref{corollaryeigreduced} (considering $r=0,1$) reveals the expression for eigenvalues and eigenvectors of the discretized Laplace eigenvalue problem \eqref{pd1} with Dirichlet boundary conditions \eqref{Dbc}.
\begin{corollary}\label{coroloryclosedformreduced}
Let $p,n\in\mathbb{N}^{\star}$ such that $p$ even and $n\geq p+\frac{p}{2}$. Then, the eigenpairs of the matrix
$\overline{L}_n^{p,0}=\left(\overline{M}_n^{p,0}\right)^{-1}\overline{K}_n^{p,0}$ can be expressed as
\begin{equation*}
\lambda_j\left(\overline{L}_n^{p,0}\right)=n^{2}\frac{g_p^{1}\left(\frac{j\pi}{n}\right)}{g_p^{0}\left(\frac{j\pi}{n}\right)},\quad j=1,\ldots, n,
\end{equation*}
and
\begin{equation*}
u_{i,j}\left(\overline{L}_n^{p,0}\right)=\sqrt{\frac{2}{n}}\;c_j\sin\left(\frac{j\pi}{n}\left(i-\frac{1}{2}\right)\right),\quad  c_j=\left\{\begin{array}{ll}
    \frac{1}{\sqrt{2}},& j=n, \\
    1,& \text{otherwise},
\end{array}
\right.\quad  i,j=1,\ldots, n,
\end{equation*}
where $g_p^{r}$ is given by \eqref{g_r} for $r=0,1$.
\end{corollary}

\begin{remark}\label{rmk:outlierfreereduced}
The explicit closed-form expressions, described in Corollary~\ref{coroloryclosedformreduced}, for the approximated spectrum of the Laplace operator with Dirichlet boundary conditions show that also the reduced spline spaces $\overline{\mathbb{S}}_{p,n,0}$ ($p$ even) lead to outlier-free discretizations.
This follows from \eqref{exactDbc}, using the same line of arguments as in Remark~\ref{rmk:outlierfree0}. Indeed, by setting $\theta_j= \frac{j\pi}{n}$, $j=1, \ldots, n$, we get
\begin{equation*}
\dfrac{\lambda_j\left(\overline L_n^{p,0}\right)-(j\pi)^2}{(j\pi)^2}=
\dfrac{n^2}{(j\pi)^2}(e_p(\theta_j)-(\theta_j)^2)=\dfrac{e_p(\theta_j)-(\theta_j)^2}{(\theta_j)^2}.
\end{equation*}
Thus, for all $j=1,\ldots,n$, from \eqref{error-bound} we deduce that the relative error between $\lambda_j\left(\overline L_n^{p,0}\right)$ and the $j$-th eigenvalue of the Laplace operator with Dirichlet boundary conditions converges to zero as $p$ increases, resulting in an outlier-free approximation.
\end{remark}

\begin{remark}
From Corollaries~\ref{coroloryclosedformreduced} and~\ref{coroloryclosedform0} we see that, when considering the spaces $\overline{\mathbb{S}}_{p,n+1,0}$ and $\mathbb{S}_{p,n,0}^{\opt}$ with $p$ even, they both yield $n$ identical eigenvalues for the discretized Laplace operator -- note that the former space has an additional $(n+1)$-st eigenvalue. However, the corresponding eigenvectors differ consistently.
For the reduced spline space, there is a shift of $-1/2$ compared to the optimal spline subspace. This shift of $-1/2$ is not introduced in the construction of the basis of the reduced spline space \eqref{basisrpeven} but is present in the basis of the optimal spline subspace \eqref{bsplines0even}.
\end{remark}

\subsection{Multidimensional extension}\label{sec:multidimclosedform}
Now, we expand the closed-form expressions for eigenvalues and eigenvectors from the $1$-dimensional to the $d$-dimensional context. We start by introducing the $d$-dimensional eigenvalue problem and outline the approximate variational formulation.
We focus on the tensor product of optimal spline subspaces suited for the Dirichlet boundary conditions \eqref{Dbc}. Similar results can be obtained for the other previously mentioned boundary conditions and spaces, including the reduced spline spaces, and we refer the reader to the general setting presented in \cite{ekstrom2018eigenvalues}.

Let $d\in\mathbb{N}^{\star}$, we consider the following $d-$dimensional Laplace eigenvalue problem:
\begin{equation}\label{pd2}
\left\{
\begin{array}{ll}
-\Delta  u=\lambda u,\quad (0,1)^d,\\[0.1cm]
u=0,\quad \partial\left( (0,1)^d \right).
\end{array}
\right.
\end{equation}
The non-trivial exact solutions $(\lambda_{\bm{k}},u_{\bm{k}})$ of \eqref{pd2} are given by
\begin{equation*}
\lambda_{\bm{k}}=\displaystyle\sum_{s=1}^{d}\left(k_s\pi\right)^2,\quad u_{\bm{k}}(\bm{x})=\displaystyle\prod_{s=1}^d \sin\left( k_s \pi x_s\right),
\end{equation*}
for $\bm{k}=(k_1,\ldots,k_d)\in(\mathbb{N}^{\star})^d$ and $\bm{x}=(x_1,\ldots,x_d)\in [0,1]^d$.

Given $\bm{n}=(n_1,\ldots,n_d)\in(\mathbb{N}^{\star})^d$ and $\bm{p}=(p_1,\ldots,p_d)\in(\mathbb{N}^{\star})^d$, let us consider the tensor-product space $\bigotimes_{s=1}^{d} \mathbb{S}_{p_s,n_s,0}^{\opt}$
in the context of the Galerkin method to discretize \eqref{pd2}. 
From Section~\ref{sec:DBCBasis}, it is clear that
\begin{equation*}
N_{\bm{i},0}^{\bm{p}}=N_{i_1,0}^{p_1}\otimes\dots\otimes N_{i_d,0}^{p_d},\quad \bm{i}=\bm{1},\ldots,\bm{n},
\end{equation*}
forms a basis for this space, where $\bm{i}=(i_1,\ldots,i_d)\in(\mathbb{N}^{\star})^d$ is a multi-index varying between $\bm{1}=(1,\ldots,1)$ and $\bm{n}=(n_1,\ldots,n_d)$.
The approximate solutions of \eqref{pd2} can be computed as follows: 
\begin{equation}\label{multvar}
\text{Find}\quad (\lambda_{\bm{k},\bm{h}},u_{\bm{k},\bm{h}})\in\mathbb{R}\times \bigotimes_{s=1}^{d} \mathbb{S}_{p_s,n_s,0}^{\opt}
\quad \text{such that}\quad   
K( u_{\bm{k},\bm{h}},v)=\lambda_{\bm{k},\bm{h}}M(u_{\bm{k},\bm{h}},v),\quad \forall v\in  \bigotimes_{s=1}^{d} \mathbb{S}_{p_s,n_s,0}^{\opt},
\end{equation}
where
\begin{equation*}
M( u_{\bm{k},\bm{h}},v)=\displaystyle\int_{(0,1)^d} u_{\bm{k},\bm{h}} (\bm{x})  v (\bm{x})\;{\dd}\bm{x},\quad
K( u_{\bm{k},\bm{h}},v)=\displaystyle\int_{(0,1)^d}\nabla u_{\bm{k},\bm{h}} (\bm{x}) \cdot \nabla v (\bm{x})\;{\dd}\bm{x}.
\end{equation*}
The problem described by \eqref{multvar} can be reformulated as a discrete eigenvalue problem,
\begin{equation}\label{ddimsystemdis}
\left( M^{\bm{p},0}_{\bm{n}} \right)^{-1}K^{\bm{p},0}_{\bm{n}}\; \bm{u}_{\bm{k},\bm{h}}=\lambda_{\bm{k},\bm{h}} \bm{u}_{\bm{k},\bm{h}}.
\end{equation}
The matrices $M^{\bm{p},0}_{\bm{n}}$ and $K^{\bm{p},0}_{\bm{n}}$ are called the mass matrix and the stiffness matrix, respectively, and are defined by
\begin{equation*}
M^{\bm{p},0}_{\bm{n}}=\left[  \displaystyle\int_{(0,1)^d}  N_{\bm{i},0}^{\bm{p}}(\bm{x}) N_{\bm{j},0}^{\bm{p}}(\bm{x}) \;{\dd}\bm{x}\right]_{\bm{i},\bm{j}=\bm{1}}^{\bm{n}},\quad
K^{\bm{p},0}_{\bm{n}}=\left[  \displaystyle\int_{(0,1)^d} \nabla N_{\bm{i},0}^{\bm{p}}(\bm{x}) \cdot \nabla N_{\bm{j},0}^{\bm{p}}(\bm{x}) \;{\dd}\bm{x}\right]_{\bm{i},\bm{j}=\bm{1}}^{\bm{n}}.
\end{equation*}
Note that
\begin{equation}\label{generalmatricesdmim-1}
M^{\bm{p},0}_{\bm{n}}=\bigotimes_{s=1}^d M_{n_s}^{p_s,0},\quad
K^{\bm{p},0}_{\bm{n}}=\displaystyle\sum_{r=1}^d\left( \bigotimes_{s=1}^{r-1}M_{n_s}^{p_s,0}\right)\otimes K_{n_r}^{p_r,0}\otimes \left( \bigotimes_{s=r+1}^d M_{n_s}^{p_s,0}  \right),
\end{equation}
and
\begin{equation}\label{generalmatricesdmim-2}
L^{\bm{p},0}_{\bm{n}}=\left( M^{\bm{p},0}_{\bm{n}} \right)^{-1}K^{\bm{p},0}_{\bm{n}} =\displaystyle\sum_{r=1}^d\left( \bigotimes_{s=1}^{r-1}I_{n_s}\right)\otimes \left( \left( M_{n_r}^{p_r,0} \right)^{-1} K_{n_r}^{p_r,0}\right)\otimes \left( \bigotimes_{s=r+1}^d I_{n_s}  \right).
\end{equation}
The matrices $K_{n_r}^{p_r,0}$ and $M_{n_r}^{p_r,0}$ for a fixed $r\in\{1,\ldots,d\}$ are precisely those defined in \eqref{allmatrices} and studied in detail in Section~\ref{sec:closedformdiri}. For more details on the above constructions and arguments of this section, we refer the reader to \cite[Section~6.1]{ekstrom2018eigenvalues}.

In the next theorem we provide a closed-form expression for the eigenvalues and eigenvectors of the system \eqref{ddimsystemdis}, which is a direct consequence of Corollary~\ref{corollaryeigdiri}.
\begin{theorem}\label{extendclosedform}
Let $d\in\mathbb{N}^{\star}$ and $\bm{p},\bm{n}\in\left(\mathbb{N}^{\star}\right)^d$ such that $n_s\geq \max\left\{p_s+1,p_s+\left\lfloor\frac{p_s}{2}\right\rfloor-1\right\}$ for all $s\in\{1,\ldots, d\}$. Then, the eigenpairs of the matrices $M^{\bm{p},0}_{\bm{n}}$, $L^{\bm{p},0}_{\bm{n}}$, and $K^{\bm{p},0}_{\bm{n}}$ can be expressed as
\begin{align*}    
\lambda_{\bm{j}}\left(M^{\bm{p},0}_{\bm{n}}\right)&=\prod _{s=1}^d (n_s+1)^{-1}\left[ g^0_{p_s}\left(\frac{j_s\pi}{n_s+1}\right)\right],\\
\lambda_{\bm{j}}\left(L^{\bm{p},0}_{\bm{n}}\right)&= \displaystyle\sum_{r=1}^d (n_r+1)^2 \left[ \frac{g^1_{p_r}}{g^0_{p_r}}\left(\frac{j_r\pi}{n_r+1}\right)\right],\\
\lambda_{\bm{j}}\left(K^{\bm{p},0}_{\bm{n}}\right)&=\lambda_{\bm{j}}\left(M^{\bm{p},0}_{\bm{n}}\right) \lambda_{\bm{j}}\left(L^{\bm{p},0}_{\bm{n}}\right),
\end{align*}
and 
\begin{equation*}
\bm{u}_{\bm{j}}\left(M^{\bm{p},0}_{\bm{n}}\right)=\bm{u}_{\bm{j}}\left(L^{\bm{p},0}_{\bm{n}}\right)=\bm{u}_{\bm{j}}\left(K^{\bm{p},0}_{\bm{n}}\right)=\bigotimes_{s=1}^d\left[\sqrt{\frac{2}{n_s+1}}\sin\left(\frac{ij_s\pi}{n_s+1}\right)\right]_{i=1}^{n_s},
\end{equation*}
for $\bm{j}=\bm{1},\ldots,\bm{n}$. The function $g_p^{r}$ with $r=0,1$ is defined in \eqref{g_r}.
\end{theorem}
\begin{proof}
Based on Corollary~\ref{corollaryeigdiri}, for any fixed $s\in\{1,\ldots, d\}$, we have
\begin{equation*}
\begin{aligned}
M^{p_s,0}_{n_s}&=Q_{n_s}\left( (n_s+1)^{-1} \diag_{j=1,\ldots,n_s}\left[ g^0_{p_s}\left(\frac{j\pi}{n_s+1}\right)\right]\right)Q_{n_s},\\
K^{p_s,0}_{n_s}&=Q_{n_s}\left( (n_s+1) \diag_{j=1,\ldots,n_s}\left[ g^1_{p_s}\left(\frac{j\pi}{n_s+1}\right)\right]\right)Q_{n_s},
\end{aligned}
\end{equation*}
where $Q_{n_s}$ is  a symmetric and orthogonal matrix containing the eigenvectors, given by 
\begin{equation*}
Q_{n_s}=\sqrt{\frac{2}{n_s+1}}\left[\sin\left(\frac{ij\pi}{n_s+1}\right)\right]_{i,j=1}^{n_s}.
\end{equation*}
Then, from \eqref{generalmatricesdmim-1} and \eqref{generalmatricesdmim-2}, we deduce
\begin{equation*} 
M^{\bm{p},0}_{\bm{n}}=\bigotimes_{s=1}^d M_{n_s}^{p_s,0}
=\left( \bigotimes_{s=1}^d Q_{n_s} \right) \left( \bigotimes_{s=1}^d \left( (n_s+1)^{-1}\diag_{j=1,\ldots,n_s}\left[ g^0_{p_s}\left(\frac{j\pi}{n_s+1}\right)\right]\right)\right)\left( \bigotimes_{s=1}^d Q_{n_s} \right),
\end{equation*}
and
\begin{equation*} 
\begin{aligned}
L^{\bm{p},0}_{\bm{n}} &=\displaystyle\sum_{r=1}^d\left( \bigotimes_{s=1}^{r-1}I_{n_s}\right)\otimes \left( \left( M_{n_r}^{p_r,0} \right)^{-1} K_{n_r}^{p_r,0}\right)\otimes \left( \bigotimes_{s=r+1}^d I_{n_s} \right)\\
&=\displaystyle\sum_{r=1}^d\left( \bigotimes_{s=1}^{r-1}I_{n_s}\right)\otimes \left\{ Q_{n_r} \left( (n_r+1)^2\diag_{j=1,\ldots,n_r}\left[ \frac{g^1_{p_r}}{g^0_{p_r}}\left(\frac{j\pi}{n_r+1}\right)\right]\right) Q_{n_r}\right\}\otimes \left( \bigotimes_{s=r+1}^d I_{n_s} \right)\\
&= \left( \bigotimes_{s=1}^d Q_{n_s} \right) \left\{\displaystyle\sum_{r=1}^d \left( \bigotimes_{s=1}^{r-1} I_{n_s}\right) \otimes \left((n_r+1)^2 \diag_{j=1,\ldots,n_r}\left[ \frac{g^1_{p_r}}{g^0_{p_r}}\left(\frac{j\pi}{n_r+1}\right)\right]\right) \otimes \left( \bigotimes_{s=r+1}^{d} I_{n_s}\right)\right\} \left( \bigotimes_{s=1}^d Q_{n_s} \right).
\end{aligned}
\end{equation*}
Finally, we have $K^{\bm{p},0}_{\bm{n}}=M^{\bm{p},0}_{\bm{n}}L^{\bm{p},0}_{\bm{n}}$, 
which completes the proof.  
\end{proof}

\section{Conclusions, further comments, and future work}\label{sec:conclusions}
We have provided explicit closed-form expressions for the eigenvalues and eigenvectors of the mass and stiffness matrices resulting from isogeometric Galerkin discretizations in optimal spline subspaces and specific reduced spline spaces that are outlier-free for the Laplace operator, under different types of boundary conditions. Our methodology is fully algebraic: it is based on proving that the mass and stiffness matrices possess a Toeplitz-minus-Hankel or Toeplitz-plus-Hankel structure, for any degree $p$. This specific perturbation of Toeplitz matrices results in the elimination of outliers and allows us to deduce a closed form for eigenvalues and eigenvectors. 
In all the studied cases, the eigenvalues of the mass and stiffness matrices are an explicitly known (uniform) sampling of the spectral symbol of the associated Toeplitz matrices.

The results in Section~\ref{sec:closedformoptspaces} are not confined to the mass and stiffness matrices related to isogeometric Galerkin discretizations of zeroth and second order problems in the considered spline subspaces (the cases $r=0,1$). They also hold for higher values of $r$, i.e., for isogeometric Galerkin discretizations of polyharmonic problems of order $r$ (with $p\geq 2r-1$) under the so-called {Laplace boundary conditions}; see \cite[Section~2.1]{manni2023}. For these problems, the spaces in \eqref{Sopti} turn out to be optimal as well; see \cite[Section~3.1]{manni2023}. Such boundary conditions, however, are of not much practical interest; they result in solutions that coincide with those for $r=1$. Outlier-free discretizations for polyharmonic problems with more interesting boundary conditions have been proposed in \cite{manni2023} and an algebraic investigation of the spectral properties of the resulting matrices is a captivating topic for future research.

A natural question is to what extent the proposed outlier-free discretizations can be fruitfully used for addressing general problems with non-homogeneous boundary behavior. It is clear that a plain discretization in outlier-free spline subspaces generally leads to a substantial loss of approximation power compared to the corresponding full spline space  because of the additional homogeneous boundary conditions that characterize the subspaces of interest. However, for problems identified by sufficiently smooth data, a suitable data-correction process for the missing boundary derivatives has been proposed in \cite[Section~5]{manni2022application}, which is analogous to the classical reduction from non-homogeneous to homogeneous Dirichlet boundary conditions. When coupled with this boundary data correction, the discretization in outlier-free spline subspaces achieves full approximation order both in the univariate and in the multivariate tensor-product case. 
It is noteworthy that the presented spectral results are highly valuable for the development of powerful solvers in that context. For more details and insights, we direct readers to \cite[Remark~1]{ekstrom2018eigenvalues} and references therein.

Lastly, an important aspect is that the machinery used here for the outlier-free theory could extend to the case of coercive second order operators with variable coefficients. The key will be the nature of the linear positive operator (see \cite{LPO}, \cite[Corollary~6.2, Section~6.3]{garoni2017generalized}, \cite[Corollary~3.2, Section~3.3]{garoni2018generalized} and references therein) of the underlying matrices and we are convinced that designing outlier-free approximating matrices in a more general setting is worth to be investigated in future research.

\section*{Acknowledgements}
C.~Manni, S.~Serra-Capizzano, and H.~Speleers are members of the research group GNCS (Gruppo Nazionale per il Calcolo Scientifico) of INdAM (Istituto Nazionale di Alta Matematica).
C. Manni and H. Speleers acknowledge the MUR Excellence Department Project MatMod@TOV (CUP E83C23000330006) awarded to the Department of Mathematics of the University of Rome Tor Vergata. They are also partially supported by a Project of Relevant National Interest (PRIN) under the National Recovery and Resilience Plan (PNRR) funded by the European Union -- Next Generation EU (CUP E53D23017910001), by the Italian Research Center in High Performance Computing, Big Data and Quantum Computing (CUP E83C22003230001), and by a GNCS project (CUP E53C23001670001).
The work of S. Serra-Capizzano is partially supported by a PRIN-PNRR project (CUP J53D23003780006) and by a GNCS project (CUP E53C23001670001). He is also funded via the European High-Performance Computing Joint Undertaking (JU) under grant agreement No 955701. The JU receives support from the European Union’s Horizon 2020 research and innovation programme and Belgium, France, Germany, Switzerland. Furthermore, he is grateful for the support of the Laboratory of Theory, Economics and Systems -- Department of Computer Science at Athens University of Economics and Business and to the ``Como Lake Center for AstroPhysics'' of the University of Insubria.

\appendix
\section{Proofs of main theorems}\label{sec:A}
To simplify the notation, we define the parameters
\begin{equation*}
q=\left\lfloor\frac{p}{2}\right\rfloor-p, \quad
p_0=p-2\left\lfloor \frac{p}{2}\right\rfloor,
\end{equation*}
which help in unifying the different basis constructions for even and odd degree $p$.
Moreover, in the case $p=1$, any matrix block with $\left\lfloor \frac{p}{2}\right\rfloor$ rows or columns is considered to be non-existent.

\subsection{Proof of Theorem~\ref{closedformdiri}}\label{sec:A1}
Fix $p,n\in\mathbb{N}^{\star}$ and $r\in\mathbb{N}$ such that $n\geq\max\left\{p+1,p+\left\lfloor\frac{p}{2}\right\rfloor-1\right\}$ and $r\leq p$.
From the construction of the basis in \eqref{basis0podd} and \eqref{basis0peven}, and using the assumption $n\geq p+1$, we deduce 
\begin{equation*}
\begin{bmatrix}
N^p_{1,0}\\
N^p_{2,0}\\
\vdots\\
N^p_{n,0}\\
\end{bmatrix} = 
\begin{bmatrix}
\underbrace{\overbrace{\;\; L^0_{n} \;\;\bigg\vert}^{\lfloor \frac{p}{2}\rfloor+1}  \;\; I_{n}  \;\; \overbrace{   \bigg\vert \;\; R^0_{n}  \;\;   }^{\lfloor \frac{p}{2}\rfloor+1}  }_{n+p+2-p_0}
\end{bmatrix}
\begin{bmatrix}
N^p_{-p}\\
\vdots\\
N^p_{0}\\
N^p_{1}\\
\vdots\\
N^p_{n+1-p_0}\\
\end{bmatrix},
\end{equation*}
where
\begin{equation*}
L^0_{n}=\begin{bmatrix}
    -J_{\left\lfloor \frac{p}{2}\right\rfloor}& O_{\left(\left\lfloor \frac{p}{2}\right\rfloor,1\right)}\\
    O_{\left(n-\left\lfloor \frac{p}{2}\right\rfloor,\left\lfloor \frac{p}{2}\right\rfloor\right)} &  O_{\left(n-\left\lfloor \frac{p}{2}\right\rfloor,1\right)}
\end{bmatrix},\quad
R^0_{n}=\begin{bmatrix}
    O_{\left(n-\left\lfloor \frac{p}{2}\right\rfloor,1\right)}  & O_{\left(n-\left\lfloor \frac{p}{2}\right\rfloor,\left\lfloor \frac{p}{2}\right\rfloor\right)}\\
    O_{\left(\left\lfloor \frac{p}{2}\right\rfloor,1\right)} &  -J_{\left\lfloor \frac{p}{2}\right\rfloor}
\end{bmatrix}.
\end{equation*}
We recall that $I_k$ is the identity matrix, $J_k$ the exchange matrix, and $O_{(k,l)}$ the zero matrix of the indicated size. 
Through a straightforward computation, we arrive at
\begin{equation}\label{difexplbasis0}
N^p_{i,0}=\left\{
\begin{array}{ll}
-N^p_{q-i}+N^p_{q+i},& i\in\left\{1,\ldots,\left\lfloor\frac{p}{2}\right\rfloor\right\},\\[0.2cm]
N^p_{q+i},& i\in\left\{\left\lfloor\frac{p}{2}\right\rfloor+1,\ldots,n-\left\lfloor\frac{p}{2}\right\rfloor\right\},\\[0.2cm]
N^p_{q+i}-N^p_{q+2(n+1)-i},& i\in\left\{n+1-\left\lfloor\frac{p}{2}\right\rfloor,\ldots,n\right\}.
\end{array}
 \right.
\end{equation}
When we assume their restriction to the interval $[0,1]$, the basis functions in \eqref{difexplbasis0} can be written in the following general form:
\begin{equation}\label{generalformbasis0}
N^p_{i,0}=N^p_{q+i}-N^p_{q-k_i},\quad i=1,\ldots,n,
\end{equation}
where $k_i$ represents a certain index depending on the value of $i\in\{1,\ldots,n\}$. Here we exploit the fact that $N^p_{l}$ has no support on $[0,1]$ for very large positive or negative values of index $l$; see \eqref{suppandregucb}. 

Now, we proceed to compute the elements of the matrix $X^{p,r,0}$; see Definition~\ref{difmatrixX} with $b=0$. For fixed $i,j\in\{1, \ldots, n\}$, we have 
\begin{equation*}
\begin{aligned}
X^{p,r,0}_{i,j}&=\displaystyle\int_{0}^{1} (N_{i,0}^p)^{(r)}(x)(N_{j,0}^p)^{(r)}(x)\;{\dd}x\\
&=\displaystyle\int_{0}^{1} \left[(N_{q+i}^p)^{(r)}(x)-(N_{q-k_i}^p)^{(r)}(x)\right]\left[(N_{q+j}^p)^{(r)}(x)-(N_{q-k_j}^p)^{(r)}(x)\right]\;{\dd}x\\
&=\displaystyle\int_{0}^{1}(N_{q+i}^p)^{(r)}(x) (N_{q+j}^p)^{(r)}(x)\;{\dd}x +\displaystyle\int_{0}^{1}(N_{q-k_i}^p)^{(r)}(x) (N_{q-k_j}^p)^{(r)}(x)\;{\dd}x\\
&\qquad-\displaystyle\int_{0}^{1}(N_{q+i}^p)^{(r)}(x) (N_{q-k_j}^p)^{(r)}(x)\;{\dd}x
-\displaystyle\int_{0}^{1}(N_{q-k_i}^p)^{(r)}(x) (N_{q+j}^p)^{(r)}(x)\;{\dd}x,
\end{aligned}
\end{equation*}
and using the definitions of $N^{p}_{i}$ in \eqref{bsplines0odd} and \eqref{bsplines0even}, we find
\begin{equation}\label{exprelmX0}
\begin{aligned}
X^{p,r,0}_{i,j}
&=(n+1)^{2r-1}\displaystyle\int_{-i-q+\frac{1-p_0}{2}}^{n+1-i-q+\frac{1-p_0}{2}} \mathcal{N}_p^{(r)}(t) \mathcal{N}_p^{(r)}(t+i-j)\;{\dd}t\\
&\qquad+(n+1)^{2r-1}\displaystyle\int_{k_j-q+\frac{1-p_0}{2}}^{n+1+k_j-q+\frac{1-p_0}{2}} \mathcal{N}_p^{(r)}(t) \mathcal{N}_p^{(r)}(t+k_i-k_j)\;{\dd}t\\
&\qquad-(n+1)^{2r-1}\displaystyle\int_{k_j-q+\frac{1-p_0}{2}}^{n+1+k_j-q+\frac{1-p_0}{2}} \mathcal{N}_p^{(r)}(t) \mathcal{N}_p^{(r)}(t-i-k_j)\;{\dd}t\\
&\qquad-(n+1)^{2r-1}\displaystyle\int_{k_i-q+\frac{1-p_0}{2}}^{n+1+k_i-q+\frac{1-p_0}{2}} \mathcal{N}_p^{(r)}(t) \mathcal{N}_p^{(r)}(t-k_i-j)\;{\dd}t.
\end{aligned}
\end{equation}
As stipulated earlier, the indices $k_i$ and $k_j$ depend on the value of $i$ and $j$, respectively. Thus, to compute $X^{p,r,0}$ it is imperative to address all the possible cases in \eqref{difexplbasis0}.
From Lemma~\ref{anti-sy} we know that $X^{p,r,0}$ is symmetric and centrosymmetric. This observation reduces the proof of \eqref{whatweneedtoproof0} to two main steps.

\paragraph{Step 1:}
We show that the block 
\begin{equation*}
\left(X^{p,r,0}_{i,j}\right)_{\left\lfloor\frac{p}{2}\right\rfloor+1\leq i,j\leq n-\left\lfloor\frac{p}{2}\right\rfloor}
\end{equation*}
is a Toeplitz matrix.
Let us fix $i,j\in\left\{\left\lfloor\frac{p}{2}\right\rfloor+1,\ldots,n-\left\lfloor\frac{p}{2}\right\rfloor\right\}$. In this case, the expression of the matrix element in \eqref{exprelmX0} simplifies to
\begin{equation*}
X^{p,r,0}_{i,j} = (n+1)^{2r-1}\displaystyle\int_{-i-q+\frac{1-p_0}{2}}^{n+1-i-q+\frac{1-p_0}{2}} \mathcal{N}_p^{(r)}(t) \mathcal{N}_p^{(r)}(t+i-j)\;{\dd}t.
\end{equation*}
By observing that
\begin{equation*}
n+1-i-q+\frac{1-p_0}{2}\geq p+1, \quad -i-q+\frac{1-p_0}{2}\leq0,
\end{equation*}
and by using \eqref{suppandregucb}, \eqref{innerprodcb}, we obtain
\begin{equation*}
X^{p,r,0}_{i,j}
=(n+1)^{2r-1}\displaystyle\int_{0}^{p+1} \mathcal{N}_p^{(r)}(t) \mathcal{N}_p^{(r)}(t+i-j)\;{\dd}t
=(n+1)^{2r-1}(-1)^{r}\mathcal{N}_{2p+1}^{(2r)}(p+1+i-j).
\end{equation*}
This concludes the proof of the first step.

\paragraph{Step 2:}
We show that the block
\begin{equation*}
\left(X_{i,j}^{p,r,0}\right)_{\substack{
1\leq i\leq \left\lfloor \frac{p}{2}\right\rfloor\\[2pt]
i\leq j\leq n}}
\end{equation*}
has a Toeplitz-minus-Hankel form, where the Hankel part starts with the third element of the associated Toeplitz matrix. More precisely, for the corresponding indices $(i,j)$ we have
\begin{equation}\label{X0termfinal}
X^{p,r,0}_{i,j}=(n+1)^{2r-1}\left((-1)^r\mathcal{N}_{2p+1}^{(2r)}(p+1+i-j)-(-1)^r\mathcal{N}_{2p+1}^{(2r)}(p+1+i+j) \right). 
\end{equation}
Let us fix $i\in \left\{1,\ldots,\left\lfloor\frac{p}{2}\right\rfloor\right\}$. In the following, we discuss the three different cases for $j$ arising from \eqref{difexplbasis0}.

If $j\in\left\{i,\ldots,\left\lfloor\frac{p}{2}\right\rfloor\right\}$, then we have $k_i=i$ and $k_j=j$ in \eqref{generalformbasis0}. In this case, the expression of the matrix element in \eqref{exprelmX0} reads as
\begin{equation}\label{X0termfp1}
\begin{aligned}
X^{p,r,0}_{i,j}
&=(n+1)^{2r-1}\displaystyle\int_{-i-q+\frac{1-p_0}{2}}^{n+1-i-q+\frac{1-p_0}{2}} \mathcal{N}_p^{(r)}(t) \mathcal{N}_p^{(r)}(t+i-j)\;{\dd}t\\
&\qquad+(n+1)^{2r-1}\displaystyle\int_{j-q+\frac{1-p_0}{2}}^{n+1+j-q+\frac{1-p_0}{2}} \mathcal{N}_p^{(r)}(t) \mathcal{N}_p^{(r)}(t+i-j)\;{\dd}t\\
&\qquad-(n+1)^{2r-1}\displaystyle\int_{j-q+\frac{1-p_0}{2}}^{n+1+j-q+\frac{1-p_0}{2}} \mathcal{N}_p^{(r)}(t) \mathcal{N}_p^{(r)}(t-i-j)\;{\dd}t\\
&\qquad-(n+1)^{2r-1}\displaystyle\int_{i-q+\frac{1-p_0}{2}}^{n+1+i-q+\frac{1-p_0}{2}} \mathcal{N}_p^{(r)}(t) \mathcal{N}_p^{(r)}(t-i-j)\;{\dd}t.
\end{aligned}
\end{equation}
Since $n\geq p+1$, we have
\begin{equation*}
n+1\pm i-q+\frac{1-p_0}{2}\geq p+1, \quad n+1+j-q+\frac{1-p_0}{2}\geq p+1.
\end{equation*}
Together with \eqref{suppandregucb}, this implies that the equality in \eqref{X0termfp1} is still valid if the upper bounds of the four integration intervals in \eqref{X0termfp1} are replaced by $p+1$. 
By leveraging Lemma~\ref{integcal} and the identity
\begin{equation*}
p+1+q-\frac{1-p_0}{2}=-q+\frac{1-p_0}{2},
\end{equation*}
we get
\begin{equation}\label{prof1t1}
\displaystyle\int_{j-q+\frac{1-p_0}{2}}^{p+1} \mathcal{N}_p^{(r)}(t) \mathcal{N}_p^{(r)}(t+i-j)\;{\dd}t=\displaystyle\int_{-i+j}^{-i-q+\frac{1-p_0}{2}} \mathcal{N}_p^{(r)}(t) \mathcal{N}_p^{(r)}(t+i-j)\;{\dd}t,
\end{equation}
and 
\begin{equation}\label{prof1t2}
\displaystyle\int_{i-q+\frac{1-p_0}{2}}^{p+1} \mathcal{N}_p^{(r)}(t) \mathcal{N}_p^{(r)}(t-i-j)\;{\dd}t=\displaystyle\int_{i+j}^{j-q+\frac{1-p_0}{2}} \mathcal{N}_p^{(r)}(t) \mathcal{N}_p^{(r)}(t-i-j)\;{\dd}t.
\end{equation}
By substituting \eqref{prof1t1}--\eqref{prof1t2} into \eqref{X0termfp1} and by using \eqref{suppandregucb}, \eqref{innerprodcb}, we arrive at \eqref{X0termfinal}.

Now, let us consider $j\in\left\{  \left\lfloor\frac{p}{2}\right\rfloor+1,\ldots,n- \left\lfloor\frac{p}{2}\right\rfloor \right\}$. In this case, the expression of the matrix element in \eqref{exprelmX0} reads as
\begin{equation*}
\begin{aligned}
X^{p,r,0}_{i,j}
&=(n+1)^{2r-1}\displaystyle\int_{-i-q+\frac{1-p_0}{2}}^{n+1-i-q+\frac{1-p_0}{2}} \mathcal{N}_p^{(r)}(t) \mathcal{N}_p^{(r)}(t+i-j)\;{\dd}t\\
&\qquad -(n+1)^{2r-1}\displaystyle\int_{i-q+\frac{1-p_0}{2}}^{n+1+i-q+\frac{1-p_0}{2}} \mathcal{N}_p^{(r)}(t) \mathcal{N}_p^{(r)}(t-i-j)\;{\dd}t.
\end{aligned}
\end{equation*}
By following the same reasoning as in the first case and by utilizing the implication
\begin{equation*}
t\leq \pm i-q+\frac{1-p_0}{2}\ \Longrightarrow\ t\mp i-j\leq 0,
\end{equation*}
we obtain \eqref{X0termfinal}.

Lastly, we take $j$ in the range $\left\{ n+1- \left\lfloor\frac{p}{2}\right\rfloor,\ldots,n\right\}$.
In this case, the expression of the matrix element in \eqref{exprelmX0} reads as
\begin{equation}\label{X0termfp3}
\begin{aligned}
X^{p,r,0}_{i,j}
&=(n+1)^{2r-1}\displaystyle\int_{-i-q+\frac{1-p_0}{2}}^{n+1-i-q+\frac{1-p_0}{2}} \mathcal{N}_p^{(r)}(t) \mathcal{N}_p^{(r)}(t+i-j)\;{\dd}t\\
&\qquad+(n+1)^{2r-1}\displaystyle\int_{-2(n+1)+j-q+\frac{1-p_0}{2}}^{-(n+1)+j-q+\frac{1-p_0}{2}} \mathcal{N}_p^{(r)}(t) \mathcal{N}_p^{(r)}(t+i-j+2(n+1))\;{\dd}t\\
&\qquad-(n+1)^{2r-1}\displaystyle\int_{-2(n+1)+j-q+\frac{1-p_0}{2}}^{-(n+1)+j-q+\frac{1-p_0}{2}} \mathcal{N}_p^{(r)}(t) \mathcal{N}_p^{(r)}(t-i-j+2(n+1))\;{\dd}t\\
&\qquad-(n+1)^{2r-1}\displaystyle\int_{i-q+\frac{1-p_0}{2}}^{n+1+i-q+\frac{1-p_0}{2}} \mathcal{N}_p^{(r)}(t) \mathcal{N}_p^{(r)}(t-i-j)\;{\dd}t.
\end{aligned}
\end{equation}
Since $n\geq p+1$, we find
\begin{equation*}
n+1\pm i-q+\frac{1-p_0}{2}\geq p+1,\quad -2(n+1)+j-q+\frac{1-p_0}{2}\leq0.
\end{equation*}
Together with \eqref{suppandregucb}, this implies that the equality in \eqref{X0termfp3} is still valid if the upper bounds of the first and the last integration interval in \eqref{X0termfp3} are replaced by $p+1$, and the lower bounds of the second and the third integration interval in \eqref{X0termfp3} are replaced by zero. 
Furthermore, from $n\geq p+\left\lfloor\frac{p}{2}\right\rfloor-1$ we deduce
\begin{equation*}
t\geq0\ \Longrightarrow\ t\pm i-j+2(n+1)\geq p+1,
\end{equation*}
and thus the second and the third integral in \eqref{X0termfp3} are both equal zero. 
Then, by following the same reasoning as in the previous two cases,
we arrive at \eqref{X0termfinal} again.
This concludes the proof of the second step.

Finally, regrouping both steps and using the symmetry and centrosymmetry of $X^{p,r,0}$, we arrive at \eqref{whatweneedtoproof0}, and hence the proof of the theorem is complete.

\subsection{Proof of Theorem~\ref{closedformneum}}\label{sec:A2}
The proof closely mirrors the one of Theorem~\ref{closedformdiri} (see the previous section).
Let us fix the $p,n\in\mathbb{N}^{\star}$ and $r\in\mathbb{N}$ such that $r\leq p$ and
\begin{equation*}
n\geq \max\left\{2p-\left\lfloor \frac{p}{2}\right\rfloor,2p-2\left\lfloor\frac{p}{2}\right\rfloor+1\right\}=\max\left\{p+\left\lfloor\frac{p}{2}\right\rfloor+p_0,p+1+p_0\right\}.
\end{equation*}
From the construction of the basis in \eqref{basis1podd} and \eqref{basis1peven}, and using the assumption $n\geq p+1+p_0$, we have 
\begin{equation*}
\begin{bmatrix}
N^p_{1,1}\\
N^p_{2,1}\\
\vdots\\
N^p_{n,1}\\
\end{bmatrix}= 
\begin{bmatrix}
\underbrace{\overbrace{\;\; J_{n} \;\;\bigg\vert}^{\lfloor \frac{p}{2}\rfloor+p_0} \;\; I_{n} \;\; \overbrace{ \bigg\vert \;\; J_{n} \;\; }^{\lfloor \frac{p}{2}\rfloor+p_0} }_{n+p+p_0}
\end{bmatrix}
\begin{bmatrix}
N^p_{-p}\\
\vdots\\
N^p_{0}\\
N^p_{1}\\
\vdots\\
N^p_{n+p_0-1}\\
\end{bmatrix}.
\end{equation*}
Through a straightforward computation, we obtain
\begin{equation}\label{difexplbasis1}
N^p_{i,1}=\left\{
\begin{array}{ll}
N^p_{-\left\lfloor\frac{p}{2}\right\rfloor-i}+N^p_{-\left\lfloor\frac{p}{2}\right\rfloor+i-1},& i\in\left\{1,\ldots,-q\right\},\\[0.2cm]
N^p_{-\left\lfloor\frac{p}{2}\right\rfloor+i-1},& i\in\left\{-q+1,\ldots,n+q\right\},\\[0.2cm]
N^p_{-\left\lfloor\frac{p}{2}\right\rfloor+i-1}+N^p_{-\left\lfloor\frac{p}{2}\right\rfloor+2n-i}, & i\in\left\{n+q+1,\ldots,n\right\}.
\end{array}
\right.
\end{equation}
When we assume their restriction to the interval $[0,1]$, the basis functions in \eqref{difexplbasis1} can be written in the following general form:
\begin{equation}\label{generalformbasis1}
N_{i,1}^p=N^p_{-\left\lfloor\frac{p}{2}\right\rfloor+i-1}+N^p_{-\left\lfloor\frac{p}{2}\right\rfloor-k_i},
\end{equation}
where $k_i$ represents again a certain index depending on the value of $i\in\{1,\ldots,n\}$.

Now, we look into the elements of the matrix $X^{p,r,1}$; see Definition~\ref{difmatrixX} with $b=1$. 
Similar to the derivation of \eqref{exprelmX0} and taking into account the definitions of $N^{p}_{i}$ in \eqref{bsplines1odd} and \eqref{bsplines1even}, we find
\begin{equation}\label{exprelmX1}
\begin{aligned}
X^{p,r,1}_{i,j}&=n^{2r-1}\displaystyle\int_{-i+1+\left\lfloor\frac{p}{2}\right\rfloor+\frac{p_0}{2}}^{n-i+1+\left\lfloor\frac{p}{2}\right\rfloor+\frac{p_0}{2}} \mathcal{N}_p^{(r)}(t) \mathcal{N}_p^{(r)}(t+i-j)\;{\dd}t\\
&\qquad+n^{2r-1}\displaystyle\int_{k_j+\left\lfloor\frac{p}{2}\right\rfloor+\frac{p_0}{2}}^{n+k_j+\left\lfloor\frac{p}{2}\right\rfloor+\frac{p_0}{2}} \mathcal{N}_p^{(r)}(t) \mathcal{N}_p^{(r)}(t+k_i-k_j)\;{\dd}t\\
&\qquad+n^{2r-1}\displaystyle\int_{k_j+\left\lfloor\frac{p}{2}\right\rfloor+\frac{p_0}{2}}^{n+k_j+\left\lfloor\frac{p}{2}\right\rfloor+\frac{p_0}{2}} \mathcal{N}_p^{(r)}(t) \mathcal{N}_p^{(r)}(t-(i-1+k_j))\;{\dd}t\\
&\qquad+n^{2r-1}\displaystyle\int_{k_i+\left\lfloor\frac{p}{2}\right\rfloor+\frac{p_0}{2}}^{n+k_i+\left\lfloor\frac{p}{2}\right\rfloor+\frac{p_0}{2}} \mathcal{N}_p^{(r)}(t) \mathcal{N}_p^{(r)}(t-(k_i+j-1))\;{\dd}t.
\end{aligned}
\end{equation}
A key observation in the proof of Theorem~\ref{closedformdiri} was the fact that the matrix $X^{p,r,0}$ is symmetric and centrosymmetric (Lemma~\ref{anti-sy}). In the same way, one can show that the matrix $X^{p,r,1}$ is also symmetric and centrosymmetric. This reduces the proof of \eqref{whatweneedtoproof1} to two main steps.

\paragraph{Step 1:} 
We show that the block 
\begin{equation*}
\left(X^{p,r,1}_{i,j}\right)_{-q+1\leq i,j\leq n+q}
\end{equation*}
is a Toeplitz matrix.
Let $i,j\in\{-q+1,\ldots,n+q\}$. Then, we can eliminate all terms corresponding to $k_i$ and $k_j$ in \eqref{exprelmX1}, i.e.,
\begin{equation*}
X^{p,r,1}_{i,j}=n^{2r-1}\displaystyle\int_{-i+1+\left\lfloor\frac{p}{2}\right\rfloor+\frac{p_0}{2}}^{n-i+1+\left\lfloor\frac{p}{2}\right\rfloor+\frac{p_0}{2}} \mathcal{N}_p^{(r)}(t) \mathcal{N}_p^{(r)}(t+i-j)\;{\dd}t.
\end{equation*}
By observing that 
\begin{equation*}
n-i+1+\left\lfloor\frac{p}{2}\right\rfloor+\frac{p_0}{2}\geq p+1, \quad
 -i+1+\left\lfloor\frac{p}{2}\right\rfloor+\frac{p_0}{2}\leq 0,
\end{equation*}
and by using \eqref{suppandregucb}, \eqref{innerprodcb}, we obtain
\begin{equation*}
X^{p,r,1}_{i,j}
=n^{2r-1}\displaystyle\int_{0}^{p+1} \mathcal{N}_p^{(r)}(t) \mathcal{N}_p^{(r)}(t+i-j)\;{\dd}t
=n^{2r-1}(-1)^{r}\mathcal{N}_{2p+1}^{(2r)}(p+1+i-j).
\end{equation*}
This concludes the proof of the first step.

\paragraph{Step 2:}
We show that the block 
\begin{equation*}
\left(X_{i,j}^{p,r,1}\right)_{\substack{
    1\leq i\leq -q\\[2pt]
    i\leq j\leq n}}
\end{equation*}
has a Toeplitz-plus-Hankel form, where the Hankel part starts with the second element of the associated Toeplitz matrix. More precisely, for the corresponding indices $(i,j)$ we have
\begin{equation}\label{X1termfinal}
X^{p,r,1}_{i,j}=n^{2r-1}\left((-1)^r\mathcal{N}_{2p+1}^{(2r)}(p+1+i-j)+(-1)^r\mathcal{N}_{2p+1}^{(2r)}(p+1+i+j-1) \right).
\end{equation}
Let us fix $i\in\{1,\ldots,-q\}$. In the following we discuss the three different cases for $j$ arising from \eqref{difexplbasis1}.

In the first case, where we consider $j\in\{i,\ldots,-q\}$, we have $k_i=i$ and $k_j=j$ in \eqref{generalformbasis1}. Hence, the expression of the matrix element in \eqref{exprelmX1} reads as
\begin{equation}\label{X1termfp1}
\begin{aligned}
X^{p,r,1}_{i,j}&=n^{2r-1}\displaystyle\int_{-i+1+\left\lfloor\frac{p}{2}\right\rfloor+\frac{p_0}{2}}^{n-i+1+\left\lfloor\frac{p}{2}\right\rfloor+\frac{p_0}{2}} \mathcal{N}_p^{(r)}(t) \mathcal{N}_p^{(r)}(t+i-j)\;{\dd}t\\
&\qquad+n^{2r-1}\displaystyle\int_{j+\left\lfloor\frac{p}{2}\right\rfloor+\frac{p_0}{2}}^{n+j+\left\lfloor\frac{p}{2}\right\rfloor+\frac{p_0}{2}} \mathcal{N}_p^{(r)}(t) \mathcal{N}_p^{(r)}(t+i-j)\;{\dd}t\\
&\qquad+n^{2r-1}\displaystyle\int_{j+\left\lfloor\frac{p}{2}\right\rfloor+\frac{p_0}{2}}^{n+j+\left\lfloor\frac{p}{2}\right\rfloor+\frac{p_0}{2}} \mathcal{N}_p^{(r)}(t) \mathcal{N}_p^{(r)}(t-i-j+1)\;{\dd}t\\
&\qquad+n^{2r-1}\displaystyle\int_{i+\left\lfloor\frac{p}{2}\right\rfloor+\frac{p_0}{2}}^{n+i+\left\lfloor\frac{p}{2}\right\rfloor+\frac{p_0}{2}} \mathcal{N}_p^{(r)}(t) \mathcal{N}_p^{(r)}(t-i-j+1)\;{\dd}t.
\end{aligned}
\end{equation}
Since $n\geq p+1+p_0$, we have
\begin{equation*}
n\pm i+\left\lfloor\frac{p}{2}\right\rfloor+\frac{p_0}{2}\geq p+1, \quad
n+j+\left\lfloor\frac{p}{2}\right\rfloor+\frac{p_0}{2}\geq p+1.
\end{equation*}
Together with \eqref{suppandregucb}, this implies that the equality in \eqref{X1termfp1} is still valid if the upper bounds of the four integration intervals in \eqref{X1termfp1} are replaced by $p+1$. 
By exploiting Lemma~\ref{integcal} and the identity
\begin{equation*}
p-\left\lfloor\frac{p}{2}\right\rfloor-\frac{p_0}{2}=\left\lfloor\frac{p}{2}\right\rfloor+\frac{p_0}{2},
\end{equation*}
we get 
\begin{equation}\label{the2term1}
\displaystyle\int_{j+\left\lfloor\frac{p}{2}\right\rfloor+\frac{p_0}{2}}^{p+1} \mathcal{N}_p^{(r)}(t) \mathcal{N}_p^{(r)}(t+i-j)\;{\dd}t=\displaystyle\int_{-i+j}^{-i+1+\left\lfloor\frac{p}{2}\right\rfloor+\frac{p_0}{2}} \mathcal{N}_p^{(r)}(t) \mathcal{N}_p^{(r)}(t+i-j)\;{\dd}t,
\end{equation}
and 
\begin{equation}\label{the2term2}
\displaystyle\int_{i+\left\lfloor\frac{p}{2}\right\rfloor+\frac{p_0}{2}}^{p+1} \mathcal{N}_p^{(r)}(t) \mathcal{N}_p^{(r)}(t-i-j+1)\;{\dd}t=\displaystyle\int_{i+j-1}^{j+\left\lfloor\frac{p}{2}\right\rfloor+\frac{p_0}{2}} \mathcal{N}_p^{(r)}(t) \mathcal{N}_p^{(r)}(t-i-j+1)\;{\dd}t.
\end{equation}
By substituting \eqref{the2term1}--\eqref{the2term2} into \eqref{X1termfp1} and by using \eqref{suppandregucb}, \eqref{innerprodcb},
we infer \eqref{X1termfinal}.

Now, we fix $j\in\{-q+1,\ldots,n+q\}$. In this case, the expression of the matrix element in \eqref{exprelmX1} reads as 
\begin{equation*}
\begin{aligned}
X^{p,r,1}_{i,j}&=n^{2r-1}\displaystyle\int_{-i+1+\left\lfloor\frac{p}{2}\right\rfloor+\frac{p_0}{2}}^{n-i+1+\left\lfloor\frac{p}{2}\right\rfloor+\frac{p_0}{2}} \mathcal{N}_p^{(r)}(t) \mathcal{N}_p^{(r)}(t+i-j)\;{\dd}t\\
&\qquad+n^{2r-1}\displaystyle\int_{i+\left\lfloor\frac{p}{2}\right\rfloor+\frac{p_0}{2}}^{n+i+\left\lfloor\frac{p}{2}\right\rfloor+\frac{p_0}{2}} \mathcal{N}_p^{(r)}(t) \mathcal{N}_p^{(r)}(t-i-j+1))\;{\dd}t.
\end{aligned}
\end{equation*}
By following the same reasoning as in the first case and by utilizing the implication
\begin{equation*}
t\leq \pm i+\left\lfloor\frac{p}{2} \right\rfloor +\frac{p_0}{2}\ \Longrightarrow\ t\mp i-j+1\leq 0,
\end{equation*}
we obtain \eqref{X1termfinal}.

In the last case, where we consider $j\in\{n+q+1,\ldots,n\}$, the expression of the matrix element in \eqref{exprelmX1} reads as 
\begin{equation}\label{X1termfp3}
\begin{aligned}
X^{p,r,1}_{i,j}&=n^{2r-1}\displaystyle\int_{-i+1+\left\lfloor\frac{p}{2}\right\rfloor+\frac{p_0}{2}}^{n-i+1+\left\lfloor\frac{p}{2}\right\rfloor+\frac{p_0}{2}} \mathcal{N}_p^{(r)}(t) \mathcal{N}_p^{(r)}(t+i-j)\;{\dd}t\\
&\qquad+n^{2r-1}\displaystyle\int_{j-2n+\left\lfloor\frac{p}{2}\right\rfloor+\frac{p_0}{2}}^{-n+j+\left\lfloor\frac{p}{2}\right\rfloor+\frac{p_0}{2}} \mathcal{N}_p^{(r)}(t) \mathcal{N}_p^{(r)}(t+i-j+2n)\;{\dd}t\\
&\qquad+n^{2r-1}\displaystyle\int_{j-2n+\left\lfloor\frac{p}{2}\right\rfloor+\frac{p_0}{2}}^{-n+j+\left\lfloor\frac{p}{2}\right\rfloor+\frac{p_0}{2}} \mathcal{N}_p^{(r)}(t) \mathcal{N}_p^{(r)}(t-i-j+2n+1)\;{\dd}t\\
&\qquad+n^{2r-1}\displaystyle\int_{i+\left\lfloor\frac{p}{2}\right\rfloor+\frac{p_0}{2}}^{n+i+\left\lfloor\frac{p}{2}\right\rfloor+\frac{p_0}{2}} \mathcal{N}_p^{(r)}(t) \mathcal{N}_p^{(r)}(t-i-j+1)\;{\dd}t.
\end{aligned}
\end{equation}
Since $n\geq p+1+p_0$, we find
\begin{equation*}
n\pm i+\left\lfloor\frac{p}{2}\right\rfloor+\frac{p_0}{2}\geq p+1,\quad j-2n+\left\lfloor\frac{p}{2}\right\rfloor+\frac{p_0}{2}\leq0.
\end{equation*}
Together with \eqref{suppandregucb}, this implies that the equality in \eqref{X1termfp3} is still valid if the upper bounds of the first and the last integration interval in \eqref{X1termfp3} are replaced by $p+1$, and the lower bounds of the second and the third integration interval in \eqref{X1termfp3} are replaced by zero. 
Furthermore, from $n\geq p+\left\lfloor\frac{p}{2}\right\rfloor+p_0$ we deduce
\begin{equation*}
t\geq0\ \Longrightarrow\ \{t-i-j+2n+1, t+i-j+2n\}\geq p+1,
\end{equation*}
and thus the second and the third integral in \eqref{X1termfp3} are both equal zero. 
Then, by following the same reasoning as in the previous two cases,
we arrive at \eqref{X1termfinal} again.
The latter concludes the proof of the second step and consequently the proof of the theorem is complete.

\subsection{Proof of Theorem~\ref{closedformmixed}}\label{sec:A3}
Also here, the proof is similar to the one of Theorem~\ref{closedformdiri} (see Section~\ref{sec:A1}).
Let us fix $p,n\in\mathbb{N}^{\star}$ and $r\in\mathbb{N}$ such that $n\geq \max\left\{p+1,p+\left\lfloor\frac{p}{2}\right\rfloor\right\}$ and $r\leq p$. 
From the construction of the basis in \eqref{basis2podd} and \eqref{basis2peven}, and using the assumption $n\geq p+1$, we have 
\begin{equation*}
\begin{bmatrix}
N^p_{1,2}\\
N^p_{2,2}\\
\vdots\\
N^p_{n,2}\\
\end{bmatrix} = 
\begin{bmatrix}
\underbrace{\overbrace{\; -J_{n} \;\bigg\vert}^{\lfloor \frac{p}{2}\rfloor}\; O_n \;\bigg\vert \; I_{n}  \; \overbrace{   \bigg\vert \; J_{n}  \;   }^{\lfloor \frac{p}{2}\rfloor+p_0}  }_{n+p+1}
\end{bmatrix}
\begin{bmatrix}
N^p_{-p}\\
\vdots\\
N^p_{0}\\
N^p_{1}\\
\vdots\\
N^p_{n}\\
\end{bmatrix},
\end{equation*}
and we can express the basis functions in the following general form:
\begin{equation*}
N_{i,2}^p=N^p_{q+i}+\gamma_i N^p_{q-k_i},\quad i=1,\ldots,n,
\end{equation*}
where
\begin{equation}\label{exprsgammaka}
\left(\gamma_i,k_i\right)=\left\{
\begin{array}{ll}
(-1,i),  & i\in\left\{1,\ldots,\left\lfloor\frac{p}{2}\right\rfloor\right\}, \\[0.2cm]
(0,\cdot),  & i\in\left\{\left\lfloor\frac{p}{2}\right\rfloor+1,\ldots,n+q\right\}, \\[0.2cm]
( 1,i-(2n+1)),  & i\in\left\{n+q+1,\ldots,n\right\}. \\
\end{array}
\right.
\end{equation}
Now, we look into the elements of the matrix $X^{p,r,2}$; see Definition~\ref{difmatrixX} with $b=2$. 
Similar to the derivation of \eqref{exprelmX0} and taking into account the definitions of $N^{p}_{i}$ in \eqref{bsplines2odd} and \eqref{bsplines2even}, we find
\begin{equation}\label{exprelmX2}
\begin{aligned}
X_{i,j}^{p,r,2}&=\left(\frac{2n+1}{2}\right)^{2r-1}\displaystyle\int_{-i-q+\frac{1-p_0}{2}}^{\frac{2n+1}{2}-i-q+\frac{1-p_0}{2}} \mathcal{N}_p^{(r)}(t) \mathcal{N}_p^{(r)}(t+i-j)\;{\dd}t\\
&\qquad+\gamma_i \gamma_j\left(\frac{2n+1}{2}\right)^{2r-1}\displaystyle\int_{k_j-q+\frac{1-p_0}{2}}^{\frac{2n+1}{2}+k_j-q+\frac{1-p_0}{2}} \mathcal{N}_p^{(r)}(t) \mathcal{N}_p^{(r)}(t+k_i-k_j)\;{\dd}t\\
&\qquad+\gamma_j\left(\frac{2n+1}{2}\right)^{2r-1}\displaystyle\int_{k_j-q+\frac{1-p_0}{2}}^{\frac{2n+1}{2}+k_j-q+\frac{1-p_0}{2}} \mathcal{N}_p^{(r)}(t) \mathcal{N}_p^{(r)}(t-i-k_j)\;{\dd}t\\
&\qquad+\gamma_i\left(\frac{2n+1}{2}\right)^{2r-1}\displaystyle\int_{k_i-q+\frac{1-p_0}{2}}^{\frac{2n+1}{2}+k_i-q+\frac{1-p_0}{2}} \mathcal{N}_p^{(r)}(t) \mathcal{N}_p^{(r)}(t-k_i-j)\;{\dd}t.
\end{aligned}
\end{equation}
In the preceding proofs, we relied on both symmetry and centrosymmetry of the matrices $X^{p,r,0}$ and $X^{p,r,1}$. However, the matrix $X^{p,r,2}$ lacks centrosymmetry and thus, we can only leverage the symmetry of the matrix to proceed. This introduces an additional step.

\paragraph{Step 1:}
We show that the block 
\begin{equation*}
\left(X^{p,r,2}_{i,j}\right)_{\left\lfloor\frac{p}{2}\right\rfloor+1\leq i,j\leq n+q}
\end{equation*}
is a Toeplitz matrix.
Let us fix $i,j\in\left\{\left\lfloor\frac{p}{2}\right\rfloor+1,\ldots, n+q\right\}$. In this case, we have $\gamma_i=\gamma_j=0$ and the expression of the matrix element in \eqref{exprelmX2} simplifies to
\begin{equation*}
X_{i,j}^{p,r,2}=\left(\frac{2n+1}{2}\right)^{2r-1}\displaystyle\int_{-i-q+\frac{1-p_0}{2}}^{\frac{2n+1}{2}-i-q+\frac{1-p_0}{2}} \mathcal{N}_p^{(r)}(t) \mathcal{N}_p^{(r)}(t+i-j)\;{\dd}t.
\end{equation*}
By observing that 
\begin{equation*}
\frac{2n+1}{2}-i-q+\frac{1-p_0}{2}\geq p+1, \quad -i-q+\frac{1-p_0}{2}\leq 0,
\end{equation*}
and by using \eqref{suppandregucb}, \eqref{innerprodcb}, we deduce
\begin{equation*}
X_{i,j}^{p,r,2}=\left(\frac{2n+1}{2}\right)^{2r-1}\displaystyle\int_{0}^{p+1} \mathcal{N}_p^{(r)}(t) \mathcal{N}_p^{(r)}(t+i-j)\;{\dd}t
=\left(\frac{2n+1}{2}\right)^{2r-1} (-1)^{r}\mathcal{N}_{2p+1}^{(2r)}(p+1+i-j).
\end{equation*}
This concludes the proof of the first step.

\paragraph{Step 2:} 
We show that the block 
\begin{equation*}
\left(X^{p,r,2}_{i,j}\right)_{\substack{
    1\leq i\leq \left\lfloor\frac{p}{2}\right\rfloor\\[2pt]
    i\leq j\leq n}}
\end{equation*}
has a Toeplitz-minus-Hankel form, where the Hankel part starts with the third element of the associated Toeplitz matrix. This form is the same as the Dirichlet case. More precisely, for the corresponding indices $(i,j)$ we have
\begin{equation}\label{X2term2final}
X^{p,r,2}_{i,j}=\left(\frac{2n+1}{2}\right)^{2r-1}\left((-1)^r\mathcal{N}_{2p+1}^{(2r)}(p+1+i-j)-(-1)^r\mathcal{N}_{2p+1}^{(2r)}(p+1+i+j) \right). 
\end{equation}
Let us fix $i\in\left\{1,\ldots,\left\lfloor\frac{p}{2}\right\rfloor\right\}$. In the following we discuss the three different cases for $j$ arising from \eqref{exprsgammaka}.

Firstly, let us consider $j\in\left\{i,\ldots,\left\lfloor\frac{p}{2}\right\rfloor\right\}$. In this case, we have $\gamma_i=\gamma_j=-1$, $k_i=i$, and $k_j=j$. Then, the expression of the matrix element in \eqref{exprelmX2} reads as
\begin{equation}\label{X2term2fp1}
\begin{aligned}
X_{i,j}^{p,r,2}&=\left(\frac{2n+1}{2}\right)^{2r-1}\displaystyle\int_{-i-q+\frac{1-p_0}{2}}^{\frac{2n+1}{2}-i-q+\frac{1-p_0}{2}} \mathcal{N}_p^{(r)}(t) \mathcal{N}_p^{(r)}(t+i-j)\;{\dd}t\\
&\qquad+\left(\frac{2n+1}{2}\right)^{2r-1}\displaystyle\int_{j-q+\frac{1-p_0}{2}}^{\frac{2n+1}{2}+j-q+\frac{1-p_0}{2}} \mathcal{N}_p^{(r)}(t) \mathcal{N}_p^{(r)}(t+i-j)\;{\dd}t\\
&\qquad-\left(\frac{2n+1}{2}\right)^{2r-1}\displaystyle\int_{j-q+\frac{1-p_0}{2}}^{\frac{2n+1}{2}+j-q+\frac{1-p_0}{2}} \mathcal{N}_p^{(r)}(t) \mathcal{N}_p^{(r)}(t-i-j)\;{\dd}t\\
&\qquad-\left(\frac{2n+1}{2}\right)^{2r-1}\displaystyle\int_{i-q+\frac{1-p_0}{2}}^{\frac{2n+1}{2}+i-q+\frac{1-p_0}{2}} \mathcal{N}_p^{(r)}(t) \mathcal{N}_p^{(r)}(t-i-j)\;{\dd}t.
\end{aligned}
\end{equation}
Since $n\geq p+1$, we have
\begin{equation*}
\frac{2n+1}{2}\pm i-q+\frac{1-p_0}{2}\geq p+1, \quad
\frac{2n+1}{2}+j-q+\frac{1-p_0}{2}\geq p+1.
\end{equation*}
Together with \eqref{suppandregucb}, this implies that the equality in \eqref{X2term2fp1} is still valid if the upper bounds of the four integration intervals in \eqref{X2term2fp1} are replaced by $p+1$. 
By using Lemma~\ref{integcal} and the identity
\begin{equation}\label{identitymixed}
p+1+q-\frac{1-p_0}{2}=-q+\frac{1-p_0}{2},
\end{equation}
we get 
\begin{equation}\label{term1mixed}
\displaystyle\int_{j-q+\frac{1-p_0}{2}}^{p+1} \mathcal{N}_p^{(r)}(t) \mathcal{N}_p^{(r)}(t+i-j)\;{\dd}t=\displaystyle\int_{-i+j}^{-i-q+\frac{1-p_0}{2}} \mathcal{N}_p^{(r)}(t) \mathcal{N}_p^{(r)}(t+i-j)\;{\dd}t,
\end{equation}
and
\begin{equation}\label{term2mixed}
\displaystyle\int_{i-q+\frac{1-p_0}{2}}^{p+1} \mathcal{N}_p^{(r)}(t) \mathcal{N}_p^{(r)}(t-(i+j))\;{\dd}t=\displaystyle\int_{i+j}^{j-q+\frac{1-p_0}{2}} \mathcal{N}_p^{(r)}(t) \mathcal{N}_p^{(r)}(t-i-j)\;{\dd}t.
\end{equation}
By substituting \eqref{term1mixed}--\eqref{term2mixed} into \eqref{X2term2fp1} and by using \eqref{suppandregucb}, \eqref{innerprodcb}, we arrive at \eqref{X2term2final}.

Now, let us fix $j\in\left\{ \left\lfloor\frac{p}{2}\right\rfloor+1,\ldots,n+q \right\}$ and so $\gamma_j=0$. In this case, the expression of the matrix element in \eqref{exprelmX2} reads as
\begin{equation*}
\begin{aligned}
X_{i,j}^{p,r,2}&=\left(\frac{2n+1}{2}\right)^{2r-1}\displaystyle\int_{-i-q+\frac{1-p_0}{2}}^{\frac{2n+1}{2}-i-q+\frac{1-p_0}{2}} \mathcal{N}_p^{(r)}(t) \mathcal{N}_p^{(r)}(t+i-j)\;{\dd}t\\
&\qquad-\left(\frac{2n+1}{2}\right)^{2r-1}\displaystyle\int_{i-q+\frac{1-p_0}{2}}^{\frac{2n+1}{2}+i-q+\frac{1-p_0}{2}} \mathcal{N}_p^{(r)}(t) \mathcal{N}_p^{(r)}(t-i-j)\;{\dd}t.
\end{aligned}
\end{equation*}
By following the same reasoning as in the first case and by using the implication
\begin{equation*}
t\leq \pm i-q+\frac{1-p_0}{2}\ \Longrightarrow\ t\mp i-j\leq 0,
\end{equation*}
we obtain \eqref{X2term2final}.

The last case addresses $j\in\{n+q+1,\ldots,n\}$, where we have $\gamma_j=1$ and $k_j=j-(2n+1)$. Then, the expression of the matrix element in \eqref{exprelmX2} reads as
\begin{equation}\label{X2term2fp3}
\begin{aligned}
X_{i,j}^{p,r,2}&=\left(\frac{2n+1}{2}\right)^{2r-1}\displaystyle\int_{-i-q+\frac{1-p_0}{2}}^{\frac{2n+1}{2}-i-q+\frac{1-p_0}{2}} \mathcal{N}_p^{(r)}(t) \mathcal{N}_p^{(r)}(t+i-j)\;{\dd}t\\
&\qquad-\left(\frac{2n+1}{2}\right)^{2r-1}\displaystyle\int_{-(2n+1)+j-q+\frac{1-p_0}{2}}^{-\frac{2n+1}{2}+j-q+\frac{1-p_0}{2}} \mathcal{N}_p^{(r)}(t) \mathcal{N}_p^{(r)}(t+i-j+2n+1)\;{\dd}t\\
&\qquad+\left(\frac{2n+1}{2}\right)^{2r-1}\displaystyle\int_{-(2n+1)+j-q+\frac{1-p_0}{2}}^{-\frac{2n+1}{2}+j-q+\frac{1-p_0}{2}} \mathcal{N}_p^{(r)}(t) \mathcal{N}_p^{(r)}(t-i-j+2n+1)\;{\dd}t\\
&\qquad-\left(\frac{2n+1}{2}\right)^{2r-1}\displaystyle\int_{i-q+\frac{1-p_0}{2}}^{\frac{2n+1}{2}+i-q+\frac{1-p_0}{2}} \mathcal{N}_p^{(r)}(t) \mathcal{N}_p^{(r)}(t-i-j)\;{\dd}t.
\end{aligned}
\end{equation}
Since $n\geq p+1$, we find
\begin{equation*}
\frac{2n+1}{2}\pm i-q+\frac{1-p_0}{2}\geq p+1, \quad -(2n+1)+j-q+\frac{1-p_0}{2}\leq0.
\end{equation*}
Together with \eqref{suppandregucb}, this implies that the equality in \eqref{X2term2fp3} is still valid if the upper bounds of the first and the last integration interval in \eqref{X2term2fp3} are replaced by $p+1$, and the lower bounds of the second and the third integration interval in \eqref{X2term2fp3} are replaced by zero. 
Furthermore, from $n\geq p+\left\lfloor\frac{p}{2}\right\rfloor$ we deduce
\begin{equation*}
t\geq 0\ \Longrightarrow\  t\pm i-j+2n+1\geq p+1,
\end{equation*}
and thus the second and the third integral in \eqref{X2term2fp3} are both equal zero. 
Then, by following the same reasoning as in the previous two cases,
we infer \eqref{X2term2final}.
This concludes the proof of the second step.

Finally, we observe that the assumption $n\geq p+\left\lfloor\frac{p}{2}\right\rfloor$ employed at the end of the last case of the second step also implies that for $j\geq n+q+1$, the  Hankel part is zero, i.e.,
\begin{equation}\label{X2term2hankel}
(-1)^r\mathcal{N}_{2p+1}^{(2r)}(p+1+i+j)=0,\quad i=1,\ldots,\left\lfloor\frac{p}{2}\right\rfloor,\quad j=n+q+1,\ldots,n. 
\end{equation}

\paragraph{Step 3:}
This additional step is necessary by the absence of the centrosymmetry property of $X_{i,j}^{p,r,2}$.
We show that the block 
\begin{equation*}
\left(X^{p,r,2}_{i,j}\right)_{\substack{n+1+q\leq i\leq n\\ 1\leq j\leq n}}
\end{equation*}
has a Toeplitz-plus-Hankel form, where the Hankel part starts with the second element of the associated Toeplitz matrix. This form is the same as the Neumann case. More precisely, for the corresponding indices $(i,j)$ we have
\begin{equation}\label{X2term3final}
X^{p,r,2}_{i,j}=\left(\frac{2n+1}{2}\right)^{2r-1}\left((-1)^r\mathcal{N}_{2p+1}^{(2r)}(p+1+i-j)+(-1)^r\mathcal{N}_{2p+1}^{(2r)}(p+1-i-j+2n+1) \right).
\end{equation}
As before, we fix $i\in\{n+q+1,\ldots,n\}$ and we discuss the three different cases for $j$.

Firstly, we consider $j\in\{n+q+1,\ldots,n\}$. In this case, we have $\gamma_i=\gamma_j=1$, $k_i=i-(2n+1)$, and $k_j=j-(2n+1)$. Then, the expression of the matrix element in \eqref{exprelmX2} reads as
\begin{equation}\label{X2term3p1}
\begin{aligned}
X_{i,j}^{p,r,2}&=\left(\frac{2n+1}{2}\right)^{2r-1}\displaystyle\int_{-i-q+\frac{1-p_0}{2}}^{\frac{2n+1}{2}-i-q+\frac{1-p_0}{2}} \mathcal{N}_p^{(r)}(t) \mathcal{N}_p^{(r)}(t+i-j)\;{\dd}t\\
&\qquad+\left(\frac{2n+1}{2}\right)^{2r-1}\displaystyle\int_{-(2n+1)+j-q+\frac{1-p_0}{2}}^{-\frac{2n+1}{2}+j-q+\frac{1-p_0}{2}} \mathcal{N}_p^{(r)}(t) \mathcal{N}_p^{(r)}(t+i-j)\;{\dd}t\\
&\qquad+\left(\frac{2n+1}{2}\right)^{2r-1}\displaystyle\int_{-(2n+1)+j-q+\frac{1-p_0}{2}}^{-\frac{2n+1}{2}+j-q+\frac{1-p_0}{2}} \mathcal{N}_p^{(r)}(t) \mathcal{N}_p^{(r)}(t-i-j+2n+1)\;{\dd}t\\
&\qquad+\left(\frac{2n+1}{2}\right)^{2r-1}\displaystyle\int_{-(2n+1)+i-q+\frac{1-p_0}{2}}^{-\frac{2n+1}{2}+i-q+\frac{1-p_0}{2}} \mathcal{N}_p^{(r)}(t) \mathcal{N}_p^{(r)}(t-i-j+2n+1)\;{\dd}t.
\end{aligned}
\end{equation}
Taking into account $n\geq p+1$, one can check that
\begin{equation*}
-i-q+\frac{1-p_0}{2}\leq 0,\quad -(2n+1)+j-q+\frac{1-p_0}{2}\leq 0,\quad -(2n+1)+i-q+\frac{1-p_0}{2}\leq 0.
\end{equation*}
Together with \eqref{suppandregucb}, this implies that the equality in \eqref{X2term3p1} is still valid if the lower bounds of the four integration intervals in \eqref{X2term3p1} are replaced by zero. 
By leveraging Lemma~\ref{integcal} and the identity \eqref{identitymixed},
we get
\begin{equation}\label{t1s3mixed}
\displaystyle\int_{0}^{-\frac{2n+1}{2}+j-q+\frac{1-p_0}{2}} \mathcal{N}_p^{(r)}(t) \mathcal{N}_p^{(r)}(t+i-j)\;{\dd}t=\displaystyle\int_{\frac{2n+1}{2}-i-q+\frac{1-p_0}{2}}^{p+1-i+j} \mathcal{N}_p^{(r)}(t) \mathcal{N}_p^{(r)}(t+i-j)\;{\dd}t,
\end{equation}
and
\begin{equation}\label{t2s3mixed}
\displaystyle\int_{0}^{-\frac{2n+1}{2}+i-q+\frac{1-p_0}{2}} \mathcal{N}_p^{(r)}(t) \mathcal{N}_p^{(r)}(t-i-j+2n+1)\;{\dd}t=\displaystyle\int_{-\frac{2n+1}{2}+j-q+\frac{1-p_0}{2}}^{p+1-(2n+1)+i+j} \mathcal{N}_p^{(r)}(t) \mathcal{N}_p^{(r)}(t-i-j+2n+1)\;{\dd}t.
\end{equation}
By substituting equations \eqref{t1s3mixed}--\eqref{t2s3mixed} into \eqref{X2term3p1} and by using \eqref{suppandregucb}, \eqref{innerprodcb},
we infer \eqref{X2term3final}.

Now, let us consider the case $j\in\left\{ \left\lfloor\frac{p}{2}\right\rfloor+1,\ldots,n+q \right\}$, where $\gamma_j=0$. Then, the expression of the matrix element in \eqref{exprelmX2} reads as
\begin{equation*}
\begin{aligned}
X_{i,j}^{p,r,2}&=\left(\frac{2n+1}{2}\right)^{2r-1}\displaystyle\int_{-i-q+\frac{1-p_0}{2}}^{\frac{2n+1}{2}-i-q+\frac{1-p_0}{2}} \mathcal{N}_p^{(r)}(t) \mathcal{N}_p^{(r)}(t+i-j)\;{\dd}t\\
&\qquad+\left(\frac{2n+1}{2}\right)^{2r-1}\displaystyle\int_{-(2n+1)+i-q+\frac{1-p_0}{2}}^{-\frac{2n+1}{2}+i-q+\frac{1-p_0}{2}} \mathcal{N}_p^{(r)}(t) \mathcal{N}_p^{(r)}(t-i-j+2n+1)\;{\dd}t.
\end{aligned}
\end{equation*}
By following the same reasoning as in the first case and by using the implication
\begin{equation*}
t\geq \frac{2n+1}{2}\pm i-q+\frac{1-p_0}{2} \ \Longrightarrow\ t\mp i-j\geq p+1,
\end{equation*}
we obtain \eqref{X2term3final}.

For the last case, where $j\in\left\{1,\ldots,\left\lfloor\frac{p}{2}\right\rfloor\right\}$,  we exploit the symmetry of the matrix, the result of Step 2, and the observation in \eqref{X2term2hankel}, to deduce that
\begin{equation*}
X^{p,r,2}_{i,j}=\left(\frac{2n+1}{2}\right)^{2r-1}(-1)^r\mathcal{N}_{2p+1}^{(2r)}(p+1+i-j).
\end{equation*}
Note that, since $n\geq p+\left\lfloor\frac{p}{2}\right\rfloor$, this expression is equal to \eqref{X2term3final} for $j\leq\left\lfloor\frac{p}{2}\right\rfloor$.
This concludes the proof of the third step and consequently the proof of the theorem is complete.

\section{Matrix algebras}\label{sec:B}
In this appendix, we give a compact overview of $\tau$ matrix algebras, focusing specifically on those encountered in our analysis. Additional algebras can be found in \cite{bozzo1995use} and \cite{di1995matrix}.

Following \cite{bozzo1995use} (see also \cite{ekstrom2018eigenvalues}), we consider the following class of real symmetric matrices
\begin{equation*}
T_n(\varepsilon,\varphi)=\begin{bmatrix}
\varepsilon&1&0&\cdots&0  \\
1&0&\ddots&\ddots&\vdots\\
0&\ddots&\ddots&\ddots&0\\
\vdots &\ddots&\ddots&0&1\\
0&\cdots&0&1&\varphi\\
\end{bmatrix},\quad \varepsilon,\varphi\in\{-1,0,1\}.
\end{equation*}
Then, the matrix algebras generated over $\mathbb{R}$ by $T_n(\varepsilon,\varphi)$ are defined as
\begin{equation*}
\tau_n(\varepsilon, \varphi)=\left\{ Q_n(\varepsilon,\varphi) \left(\diag_{j=1,\ldots,n}[a_j]\right)Q_n^T(\varepsilon,\varphi):\;  (a_j)_{1\leq j\leq n}\in\mathbb{R}^n  \right\},
\end{equation*}
where $Q_n(\varepsilon,\varphi)$ represents an orthogonal matrix determined through the diagonalization of $T_n(\varepsilon,\varphi)$:
\begin{equation*}
T_n(\varepsilon,\varphi)=Q_n(\varepsilon,\varphi)\Lambda(\varepsilon,\varphi) Q_n^T(\varepsilon,\varphi).
\end{equation*}
It can be observed from our closed-form expressions of eigenvalues and eigenvectors (see Sections~\ref{sec:closedformoptspaces} and~\ref{sec:reducedspaces}) that our mass and stiffness matrices fall within one of the following algebras.
\begin{itemize}
\item Consider $\varepsilon=\varphi=0$. Then,
\begin{equation*}
Q_n(0,0)=\sqrt{\frac{2}{n+1}}\left[\sin\left(\frac{ij\pi}{n+1}\right)\right]_{i,j=1}^{n},\quad \Lambda(0,0)=\diag_{j=1,\ldots,n}\left[2\cos\left(\frac{j\pi}{n+1}\right)\right].
\end{equation*}
In this particular case, the matrix $Q_n(0,0)$ is also symmetric.
\item Consider $\varepsilon=\varphi=1$. Then,
\begin{equation*}
\begin{aligned}
Q_n(1,1)&=\sqrt{\frac{2}{n}}\left[c_j\cos\left(\frac{(j-1)\pi}{n}\left(i-\frac{1}{2}\right)\right)\right]_{i,j=1}^{n},  \quad
c_j=\left\{\begin{array}{ll}
    \frac{1}{\sqrt{2}}, & j=1, \\
    1, & \text{otherwise},
\end{array}
\right.  \\
\Lambda(1,1)&=\diag_{j=1,\ldots,n}\left[2\cos\left(\frac{(j-1)\pi}{n}\right)\right].
\end{aligned}
\end{equation*}
\item Consider $\varepsilon=0$ and $\varphi=1$. Then,
\begin{equation*}
Q_n(0,1)=\sqrt{\frac{4}{2n+1}}\left[\sin\left(\frac{i(2j-1)\pi}{2n+1}\right)\right]_{i,j=1}^{n},\quad \Lambda(0,1)=\diag_{j=1,\ldots,n}\left[2\cos\left(\frac{(2j-1)\pi}{2n+1}\right)\right].
\end{equation*}
\item Consider $\varepsilon=\varphi=-1$. Then,
\begin{equation*}
\begin{aligned}
Q_n(-1,-1)&=\sqrt{\frac{2}{n}}\left[c_j\sin\left(\frac{j\pi}{n}\left(i-\frac{1}{2}\right)\right)\right]_{i,j=1}^{n},  \quad
c_j=\left\{\begin{array}{ll}
    \frac{1}{\sqrt{2}}, & j=n, \\
    1, & \text{otherwise},
\end{array}
\right.\\
\Lambda(-1,-1)&=\diag_{j=1,\ldots,n}\left[2\cos\left(\frac{j\pi}{n}\right)\right].
\end{aligned}
\end{equation*}
\end{itemize} 

For additional properties of elements within the matrix algebras $\tau_n(\varepsilon, \varphi)$ in terms of sine and cosine transforms, as well as the computational cost of the matrix-vector product $Q_n(\varepsilon,\varphi) \bm{v}$ for a given vector $\bm{v}$ in a preconditioning setting, we refer the reader to \cite{benedetto2000optimal} along with the references cited therein.


\begin{thebibliography}{10}

\bibitem{bapat1991hypergroups}
R.~B. Bapat and V.~S. Sunder.
\newblock On hypergroups of matrices.
\newblock {\em Linear and Multilinear Algebra}, 29:125--140, 1991.

\bibitem{Ba2}
G.~Barbarino, C.~Garoni, and S.~Serra-Capizzano.
\newblock Block generalized locally {T}oeplitz sequences: Theory and
  applications in the multidimensional case.
\newblock {\em Electronic Transactions on Numerical Analysis}, 53:113--216,
  2020.

\bibitem{bini1983spectral}
D.~Bini and M.~Capovani.
\newblock Spectral and computational properties of band symmetric {T}oeplitz
  matrices.
\newblock {\em Linear Algebra and its Applications}, 52:99--126, 1983.

\bibitem{bolten2023note}
M.~Bolten, S.-E. Ekstr{\"o}m, I.~Furci, and S.~Serra-Capizzano.
\newblock A note on the spectral analysis of matrix sequences via {GLT}
  momentary symbols: From all-at-once solution of parabolic problems to
  distributed fractional order matrices.
\newblock {\em Electronic Transactions on Numerical Analysis}, 58:136--163,
  2023.

\bibitem{bozzo1995use}
E.~Bozzo and C.~Di~Fiore.
\newblock On the use of certain matrix algebras associated with discrete
  trigonometric transforms in matrix displacement decomposition.
\newblock {\em SIAM Journal on Matrix Analysis and Applications}, 16:312--326,
  1995.

\bibitem{Bressan:2019}
A.~Bressan and E.~Sande.
\newblock Approximation in {FEM}, {DG} and {IGA}: A theoretical comparison.
\newblock {\em Numerische Mathematik}, 143:923--942, 2019.

\bibitem{cantoni1976eigenvalues}
P.~Butler and A.~Cantoni.
\newblock Eigenvalues and eigenvectors of symmetric centrosymmetric matrices.
\newblock {\em Linear Algebra and its Applications}, 13:275--288, 1976.

\bibitem{Cottrell:2009}
J.~A. Cottrell, T.~J.~R. Hughes, and Y.~Bazilevs.
\newblock {\em Isogeometric Analysis: Toward Integration of {CAD} and {FEA}}.
\newblock John Wiley \& Sons, Chichester, 2009.

\bibitem{Cottrell:2006}
J.~A. Cottrell, A.~Reali, Y.~Bazilevs, and T.~J.~R. Hughes.
\newblock Isogeometric analysis of structural vibrations.
\newblock {\em Computer Methods in Applied Mechanics and Engineering},
  195:5257--5296, 2006.

\bibitem{de1978practical}
C.~de~Boor.
\newblock {\em A Practical Guide to Splines}, volume~27.
\newblock Springer-Verlag, New York, 1978.

\bibitem{deng2021analytical}
Q.~Deng.
\newblock Analytical solutions to some generalized and polynomial eigenvalue
  problems.
\newblock {\em Special Matrices}, 9:240--256, 2021.

\bibitem{deng2021boundary}
Q.~Deng and V.~M. Calo.
\newblock A boundary penalization technique to remove outliers from
  isogeometric analysis on tensor-product meshes.
\newblock {\em Computer Methods in Applied Mechanics and Engineering},
  383:113907, 2021.

\bibitem{benedetto2000optimal}
F.~Di~Benedetto and S.~Serra-Capizzano.
\newblock Optimal multilevel matrix algebra operators.
\newblock {\em Linear and Multilinear Algebra}, 48:35--66, 2000.

\bibitem{di1995matrix}
C.~Di~Fiore and P.~Zellini.
\newblock Matrix decompositions using displacement rank and classes of
  commutative matrix algebras.
\newblock {\em Linear Algebra and its Applications}, 229:49--99, 1995.

\bibitem{DiVona:2019}
E.~Di~Vona.
\newblock {K}olmogorov {$n$}-width e spazi spline ottimi.
\newblock {\em Tesi di Laura Magistrale, Universit\`a degli Studi di {R}oma
  {T}or {V}ergata}, 2019.

\bibitem{donatelli2016spectral}
M.~Donatelli, C.~Garoni, C.~Manni, S.~Serra-Capizzano, and H.~Speleers.
\newblock Spectral analysis and spectral symbol of matrices in isogeometric
  collocation methods.
\newblock {\em Mathematics of Computation}, 85:1639--1680, 2016.

\bibitem{SINUM}
M.~Donatelli, C.~Garoni, C.~Manni, S.~Serra-Capizzano, and H.~Speleers.
\newblock Symbol-based multigrid methods for {G}alerkin {B}-spline isogeometric
  analysis.
\newblock {\em SIAM Journal on Numerical Analysis}, 55:31--62, 2017.

\bibitem{ekstrom2018eigenvalues}
S.-E. Ekstr{\"o}m, I.~Furci, C.~Garoni, C.~Manni, S.~Serra-Capizzano, and
  H.~Speleers.
\newblock Are the eigenvalues of the {B}-spline isogeometric analysis
  approximation of {$-\Delta u= \lambda u$} known in almost closed form?
\newblock {\em Numerical Linear Algebra with Applications}, 25:e2198, 2018.

\bibitem{Floater:2017}
M.~S. Floater and E.~Sande.
\newblock Optimal spline spaces of higher degree for {$L^2$} {$n$}-widths.
\newblock {\em Journal of Approximation Theory}, 216:1--15, 2017.

\bibitem{Floater:2018}
M.~S. Floater and E.~Sande.
\newblock Optimal spline spaces for {$L^2$} {$n$}-width problems with boundary
  conditions.
\newblock {\em Constructive Approximation}, 50:1--18, 2019.

\bibitem{garoni2014spectrum}
C.~Garoni, C.~Manni, F.~Pelosi, S.~Serra-Capizzano, and H.~Speleers.
\newblock On the spectrum of stiffness matrices arising from isogeometric
  analysis.
\newblock {\em Numerische Mathematik}, 127:751--799, 2014.

\bibitem{garoni2022}
C.~Garoni, C.~Manni, F.~Pelosi, and H.~Speleers.
\newblock Spectral analysis of matrices resulting from isogeometric immersed
  methods and trimmed geometries.
\newblock {\em Computer Methods in Applied Mechanics and Engineering},
  400:115551, 2022.

\bibitem{garoni2020NURBS}
C.~Garoni, C.~Manni, S.~Serra-Capizzano, and H.~Speleers.
\newblock {NURBS} in isogeometric discretization methods: A spectral analysis.
\newblock {\em Numerical Linear Algebra with Applications}, 27:e2318, 2020.

\bibitem{garoni2017generalized}
C.~Garoni and S.~Serra-Capizzano.
\newblock {\em Generalized Locally {T}oeplitz Sequences: Theory and
  Applications}, volume~1.
\newblock Springer, Cham, 2017.

\bibitem{garoni2018generalized}
C.~Garoni and S.~Serra-Capizzano.
\newblock {\em Generalized Locally {T}oeplitz Sequences: Theory and
  Applications}, volume~2.
\newblock Springer, Cham, 2018.

\bibitem{SIMAX-Q-p-d}
C.~Garoni, S.~Serra-Capizzano, and D.~Sesana.
\newblock Spectral analysis and spectral symbol of {$d$}-variate {$Q_p$}
  {L}agrangian {FEM} stiffness matrices.
\newblock {\em SIAM Journal on Matrix Analysis and Applications},
  36:1100--1128, 2015.

\bibitem{garoni2019symbol}
C.~Garoni, H.~Speleers, S.-E. Ekstr{\"o}m, A.~Reali, S.~Serra-Capizzano, and
  T.~J.~R. Hughes.
\newblock Symbol-based analysis of finite element and isogeometric {B}-spline
  discretizations of eigenvalue problems: Exposition and review.
\newblock {\em Archives of Computational Methods in Engineering},
  26:1639--1690, 2019.

\bibitem{hiemstra2021removal}
R.~R. Hiemstra, T.~J.~R. Hughes, A.~Reali, and D.~Schillinger.
\newblock Removal of spurious outlier frequencies and modes from isogeometric
  discretizations of second-and fourth-order problems in one, two, and three
  dimensions.
\newblock {\em Computer Methods in Applied Mechanics and Engineering},
  387:114115, 2021.

\bibitem{Hughes:2005}
T.~J.~R. Hughes, J.~A. Cottrell, and Y.~Bazilevs.
\newblock Isogeometric analysis: {CAD}, finite elements, {NURBS}, exact
  geometry and mesh refinement.
\newblock {\em Computer Methods in Applied Mechanics and Engineering},
  194:4135--4195, 2005.

\bibitem{Hughes:2014}
T.~J.~R. Hughes, J.~A. Evans, and A.~Reali.
\newblock Finite element and {NURBS} approximations of eigenvalue,
  boundary-value, and initial-value problems.
\newblock {\em Computer Methods in Applied Mechanics and Engineering},
  272:290--320, 2014.

\bibitem{Lyche:2018}
T.~Lyche, C.~Manni, and H.~Speleers.
\newblock Foundations of spline theory: {B}-splines, spline approximation, and
  hierarchical refinement.
\newblock In T.~Lyche, C.~Manni, and H.~Speleers, editors, {\em Splines and
  {PDE}s: From Approximation Theory to Numerical Linear Algebra}, volume 2219
  of {\em Lecture Notes in Mathematics}, pages 1--76. Springer, Cham, 2018.

\bibitem{manni2022application}
C.~Manni, E.~Sande, and H.~Speleers.
\newblock Application of optimal spline subspaces for the removal of spurious
  outliers in isogeometric discretizations.
\newblock {\em Computer Methods in Applied Mechanics and Engineering},
  389:114260, 2022.

\bibitem{manni2023}
C.~Manni, E.~Sande, and H.~Speleers.
\newblock Outlier-free spline spaces for isogeometric discretizations of
  biharmonic and polyharmonic eigenvalue problems.
\newblock {\em Computer Methods in Applied Mechanics and Engineering},
  417:116314, 2023.

\bibitem{Sande:2019}
E.~Sande, C.~Manni, and H.~Speleers.
\newblock Sharp error estimates for spline approximation: Explicit constants,
  {$n$}-widths, and eigenfunction convergence.
\newblock {\em Mathematical Models and Methods in Applied Sciences},
  29:1175--1205, 2019.

\bibitem{Sande:2020}
E.~Sande, C.~Manni, and H.~Speleers.
\newblock Explicit error estimates for spline approximation of arbitrary
  smoothness in isogeometric analysis.
\newblock {\em Numerische Mathematik}, 144:889--929, 2020.

\bibitem{schumaker2007spline}
L.~L. Schumaker.
\newblock {\em Spline Functions: Basic Theory}.
\newblock Cambridge University Press, Cambridge, 2007.

\bibitem{LPO}
S.~Serra-Capizzano.
\newblock Some theorems on linear positive operators and functionals and their
  applications.
\newblock {\em Computers \& Mathematics with Applications}, 39:139--167, 2000.

\bibitem{sogn2019robust}
J.~Sogn and S.~Takacs.
\newblock Robust multigrid solvers for the biharmonic problem in isogeometric
  analysis.
\newblock {\em Computers \& Mathematics with Applications}, 77:105--124, 2019.

\bibitem{takacs2016approximation}
S.~Takacs and T.~Takacs.
\newblock Approximation error estimates and inverse inequalities for
  {B}-splines of maximum smoothness.
\newblock {\em Mathematical Models and Methods in Applied Sciences},
  26:1411--1445, 2016.

\end{thebibliography}

\end{document}